\documentclass[12pt]{amsart}
\usepackage{mathrsfs}

\font\bbbld=msbm10 scaled\magstephalf

\usepackage{hyperref}

\usepackage{graphicx}

\usepackage{xcolor}

\newcommand{\bpartial}{\bar{\partial}}

\def \a{\alpha}
\def \p{\partial}

\newcommand{\bfC}{\hbox{\bbbld C}}

\newcommand{\bfR}{\hbox{\bbbld R}}

\newcommand{\ol}{\overline}

\newtheorem{theorem}{Theorem}[section]
\newtheorem{lemma}[theorem]{Lemma}

\newtheorem{corollary}[theorem]{Corollary}

 \theoremstyle{definition}
\newtheorem{definition}[theorem]{Definition}

\theoremstyle{remark}
\newtheorem{remark}[theorem]{Remark}

\numberwithin{equation}{section}

\begin{document}
\setlength{\baselineskip}{1.2\baselineskip}

\title[Flow for Generalized Complex Monge-Amp\`ere Type Equations]
{Parabolic Flow for Generalized complex Monge-Amp\`ere type equations 
}

\author{Wei Sun}

\address{Shanghai Center for Mathematical Sciences, Fudan University, Shanghai, China}
\email{sunwei\_math@fudan.edu.cn}

\begin{abstract}
We study the parabolic flow for generalized complex Monge-Amp\`ere type equations on closed Hermitian manifolds. We derive {\em a priori} $C^\infty$ estimates for normalized solutions, and then prove the $C^\infty$ convergence.
\end{abstract}

\maketitle

\section{Introduction}
\label{introduction}
\setcounter{equation}{0}
\medskip

Let $(M,\omega)$ be a compact Hermitian manifold of complex dimension $n\geq 2$ and $\chi$ another smooth Hermitian metric. We are interested in looking for a Hermitian metric $\chi'$ in the class $[\chi]$ satisfying
\begin{equation}
\label{introduction:elliptic-equation-origin}
	\chi'^n = \psi \sum^n_{\a = 1} b_\a \chi'^{n - \a} \wedge \omega^\a ,
\end{equation}
where $\psi$ is a smooth positive real function on $M$, $b_\a$'s are nonnegative real constants, and $\sum^n_{\a = 1} b_\a > 0$. This type of equations contains some of the most important ones in complex geometry and analysis.

The most known case is probably the complex Monge-Amp\`ere equation,
\begin{equation}
	\chi'^n = \psi \omega^n.
\end{equation}
On closed K\"ahler manifolds, this equation is equivalent to the Calabi conjecture~\cite{Calabi56, Calabi57}. In the famous work of Yau~\cite{Yau78} (see also~\cite{Aubin78}), he proved the conjecture by solving the complex Monge-Amp\`ere equation. A corollary of the Calabi conjecture is the existence of K\"ahler-Einstein metrics when the first Chern class is zero. Later, the result was extended to closed Hermitian manifolds, which was done by Cherrier~\cite{Cherrier87}, Tosatti and Weinkove~\cite{TWv10a, TWv10b}. In \cite{TWv10b}, Tosatti and Weinkove stated a Hermitian version of Calabi conjecture, which used the first Bott-Chern class instead.

All of these works are based on continuity method, while the parabolic flow method is also a powerful measure. Cao~\cite{Cao85} used K\"ahler-Ricci flow to reproduce the result of Yau~\cite{Yau78}, and Gill~\cite{Gill11} introduced Chern-Ricci flow to reproduce that of Tosatti and Weinkove~\cite{TWv10a, TWv10b}. 

Another interesting case is Donaldson's problem,
\begin{equation}
	\chi'^n = \psi \chi'^{n - 1} \wedge \omega.
\end{equation}
It was first proposed by Donaldson~\cite{Donaldson99a} in the study of moment maps. Chen~\cite{Chen00b} found out the same equation when he study the Mabuchi energy on Fano manifolds. The solvability implies a lower bound of the Mabuchi energy, and $\chi'$ is the critical metric. When $\psi$ is a constant, the equation was studied by Chen~\cite{Chen00b, Chen04}, Weinkove~\cite{Weinkove04, Weinkove06}, Song and Weinkove~\cite{SW08} using the $J$-flow.
These results were extended by Fang, Lai and Ma~\cite{FLM11} also by a parabolic flow. In \cite{FL12, FL13}, Fang and Lai studied some other parabolic flows, one of which was adopted by the author~\cite{Sun2013p} to study more general cases. 

For Donaldson's problem, continuity method seems harder to apply. The author~\cite{Sun2013e} introduced piecewise continuity method to solve the complex Monge-Amp\`ere type equations on Hermitian manifolds, which contains Donaldson's equation. Li, Shi and Yao~\cite{LiShiYao2013} reproduced the result of Song and Weinkove~\cite{SW08} by constructing a class of K\"ahler metrics. Admitting the result of Yau~\cite{Yau78}, Collins and Sz\'ekelyhidi~\cite{CollinsSzekelyhidi2014a} were also able to solve the equation by a different argument. In the argument, they construct a class of equations in form~\eqref{introduction:elliptic-equation-origin} to reach Donaldson's equation from the complex Monge-Amp\`ere equation. The approach was applied in more general equations by Sz\'ekelyhidi~\cite{Szekelyhidi2014b}.

For general cases of equation~\eqref{introduction:elliptic-equation-origin}, the solvability was conjectured by Fang, Lai and Ma~\cite{FLM11}. They solved the equation when the coefficients satisfies some special conditions. To investigate the numerical cone condition for $J$-flow conjectured in \cite{LejmiSzekelyhidi13}, Collins and Sz\'ekelyhidi~\cite{CollinsSzekelyhidi2014a} studied equation~\eqref{introduction:elliptic-equation-origin}.
Later, the author~\cite{Sun2014g} generalized the result to more general cases by piecewise continuity method.

In local coordinate charts, we may write $\omega$ and $\chi$ as
\begin{equation}
	\omega = \frac{\sqrt{-1}}{2} \sum_{i,j} g_{i\bar j} d z^i \wedge d\bar z^j
\end{equation}
and
\begin{equation}
	\chi = \frac{\sqrt{-1}}{2} \sum_{i,j} \chi_{i\bar j} d z^i \wedge d\bar z^j.
\end{equation}
For convenience, we denote
\begin{equation}
	\chi_u = \chi + \frac{\sqrt{-1}}{2} \p\bpartial u = \frac{\sqrt{-1}}{2} \sum_{i,j} (\chi_{i\bar j} + u_{i\bar j}) d z^i \wedge d\bar z^j.
\end{equation}
Therefore, equation~\eqref{introduction:elliptic-equation-origin} can be written in the form,
\begin{equation}
\label{introduction:elliptic-equation}
	\chi^n_u = \psi \sum^n_{\a = 1} b_\a \chi^{n - \a}_u \wedge \omega^\a ,\qquad \chi_u > 0.
\end{equation}
This is a fully nonlinear elliptic equation. Since the cokernal is nontrivial, it is not solvable for all $\psi > 0$. Generally, we try to find a function $u$ and a real number $b$ such that
\begin{equation}
\label{introduction:elliptic-equation-actual}
	\left\{
	\begin{aligned}
		& \chi^n_u = e^b \psi \sum^n_{\a = 1} b_\a \chi^{n - \a}_u \wedge \omega^\a, \\
		& \chi_u > 0,\qquad \sup_M u = 0 .
	\end{aligned}
	\right.
\end{equation}
Only in a few cases, we know $b = 0$ in advance. If $\chi$ and $\omega$ are both K\"ahler and $\psi$ is constant, $\psi$ is uniquely determined by
\begin{equation}
\label{introduction:definition-invariant-constant}
	\psi = c := \frac{\int_M \chi^n}{\sum^n_{\a = 1} b_\a \int_M \chi^{n - \a} \wedge \omega^\a } 
\end{equation}

To reproduce the results in \cite{Sun2014g}, we consider a parabolic flow for generalized complex Monge-Amp\`ere type equations,
\begin{equation}
\label{introduction:parabolic-flow-equation}
	\frac{\p u }{\p t} = \ln \frac{\chi^n_u}{ \sum^n_{\a = 1} b_\a \chi^{n - \a}_u \wedge \omega^\a } - \ln \psi
\end{equation}
with initial value $u (x,0) = 0$, where $\psi \in C^\infty (M)$ is positive, and $\chi_u > 0$. Following \cite{SW08, FLM11, GSun12}, we assume that there is a $C^2$ function $v$  satisfying
\begin{equation}
\label{introduction:cone-condition}
	n \chi^{n - 1}_{v} > \psi \sum^{n - 1}_{\a = 1} b_\a (n - \a) \chi^{n - \a - 1}_v \wedge \omega^\a,\,\,\text{ and } \chi_v > 0,
\end{equation}
which is called the cone condition.

In the study, the key step is probably the sharp $C^2$ estimate.
\begin{theorem}
\label{introduction:theorem-C2}
Let $(M,\omega)$ be a closed Hermitian manifold fo complex dimension $n$ and $\chi$ also a Hermitian metric. Suppose that the cone condition~\eqref{introduction:cone-condition} holds true. Then there exists a long time solution $u$ to equation~\eqref{introduction:parabolic-flow-equation}. Moreover, there are uniform constants $C$ and $A$ such that
\begin{equation}
	\Delta u + tr\chi \leq C e^{A (u - \inf_{M\times[0,t]} u)}.
\end{equation} 
\end{theorem}
\begin{remark}
As shown by Collins and Sz\'ekelyhidi~\cite{CollinsSzekelyhidi2014a}, the cone condition can be defined in the viscosity sense.
\end{remark}

The sharp $C^2$ estimate can help us to obtain $C^0$ estimate, which is probably the most difficult one in the argument. It is worth a mention that there is a direct uniform estimate for the elliptic equation by the author~\cite{Sun2014e, Sun2014u, Sun2014g} and Sz\'ekelyhidi~\cite{Szekelyhidi2014b} while that for parabolic flows is not available so far. Since there are troublesome torsion terms, we need to apply a trick due to Phong and Sturm~\cite{PhongSturm10} with making a particular perfect square (see also \cite{TWv11}). When $C^0$ and $C^2$ estimates are available, all others can be obtained by Evans-Krylov theorem~\cite{Evans82,Krylov82,Wang1992a,Wang1992b} and Schauder theory.

To reproduce the results in \cite{Sun2014g}, we need to discover some convergence property. However, it is very likely that $u(x,t)$ itself is divergent. Instead, we study the convergence of the normalized solution,
\begin{equation}
	\tilde u = u - \frac{\int_M u \omega^n}{\int_M \omega^n}.
\end{equation}

For general Hermitian manifolds, we have the following result.
\begin{theorem}
\label{introduction:theorem-hermitian}
Let $(M,\omega)$ be a closed Hermitian manifold of complex dimension $n$ and $\chi$ also a Hermitian metric. Suppose that the cone condition~\eqref{introduction:cone-condition} holds true. 
Then there exists a uniform constant $C$ such that for all time $t \geq 0$,
\begin{equation}
	\sup_{x \in M} u (x,t) - \inf_{x\in M} u(x,t) < C,
\end{equation}
given that 
\begin{equation}
\label{introduction:theorem-hermitian-condition}
	\chi^n \leq \psi \sum^n_{\a = 1} b_\a \chi^{n - \a} \wedge \omega^\a. 
\end{equation}
As a consequence, $\tilde u$ is $C^\infty$ convergent to a smooth function $\tilde u_\infty$. Moreover, there is a unique real number $b$ such that the pair $(\tilde u_\infty - \sup_M \tilde u_\infty,b)$ solves equation~\eqref{introduction:elliptic-equation-actual}.

\end{theorem}

A corollary follows from Theorem~\ref{introduction:theorem-hermitian}.
\begin{corollary}
\label{introduction:corollary-hermitian}
Let $(M,\omega)$ be a closed Hermitian manifold of complex dimension $n$ and $\chi$ also a Hermitian metric. Suppose that the cone condition~\eqref{introduction:cone-condition} holds true, and $b_1 \cdots, b_{n - 1}$ and $\psi$ are fixed. Then there is a constant $K \geq 0$ such that if $c_n \geq K$  there exists a uniform constant $C$ such that for all time $t \geq 0$,
\begin{equation}
	\sup_{x \in M} u (x,t) - \inf_{x\in M} u(x,t) < C,
\end{equation}
As a consequence, $\tilde u$ is $C^\infty$ convergent to a smooth function $\tilde u_\infty$. Moreover, there is a unique real number $b$ such that the pair $(\tilde u_\infty - \sup_M \tilde u_\infty,b)$ solves equation~\eqref{introduction:elliptic-equation-actual}. 

\end{corollary}

Should we have more knowledge of the manifold, it would be possible to obtain stronger results. When $\chi$ and $\omega$ are both K\"ahler, we have the following result.
\begin{theorem}
\label{introduction:theorem-kahler}
Let $(M,\omega)$ be a closed K\"ahler manifold of complex dimension $n$ and $\chi$ also a K\"ahler metric. Suppose that the cone condition~\eqref{introduction:cone-condition} holds true and $\psi \geq c$ for all $x\in M$, where $c$ is defined in \eqref{introduction:definition-invariant-constant}. 
Then there exists a uniform constant $C$ such that for all time $t \geq 0$,
\begin{equation}
	\sup_{x \in M} u (x,t) - \inf_{x\in M} u(x,t) < C.
\end{equation}
Consequently, $\tilde u$ is $C^\infty$ convergent to a smooth function $\tilde u_\infty$. Moreover, there is a unique real number $b$ such that the pair $(\tilde u_\infty - \sup_M \tilde u_\infty,b)$ solves equation~\eqref{introduction:elliptic-equation-actual}.

\end{theorem}

It is worth a mention that we only require the concavity of the elliptic part in several key steps. So our argument can be applied to some other parabolic flows for generalized complex Monge-Amp\`ere type equations. Fang and Lai~\cite{FL13} discussed different parabolic flows for complex Monge-Amp\`ere type equations. Collins and Sz\'ekelyhidi~\cite{CollinsSzekelyhidi2014a} used the equation
\begin{equation}
	\frac{\p u }{\p t} = - \frac{ \sum^n_{\a = 1} b_\a \chi^{n - \a}_u \wedge \omega^\a }{\chi^n_u}  .
\end{equation}
For simplicity, we shall consider an equivalent form analogous to $J$-flow when $\chi$ and $\omega$ are K\"ahler. We may call it generalized $J$-flow.
\begin{equation}
\label{introduction:parabolic-flow-equation-J}
	\frac{\p u }{\p t} = \frac{1}{c} - \frac{ \sum^n_{\a = 1} b_\a \chi^{n - \a}_u \wedge \omega^\a }{\chi^n_u}  .
\end{equation}
\begin{theorem}
\label{introduction:theorem-kahler-J}
Let $(M,\omega)$ be a closed K\"ahler manifold of complex dimension $n$ and $\chi$ also a K\"ahler metric. Suppose that the cone condition~\eqref{introduction:cone-condition} holds true for $\psi = c$. 
Then there exists a long time solution $u$ to equation~\eqref{introduction:parabolic-flow-equation-J}.
Moreover, $u$ is $C^\infty$ convergent to a smooth function $u_\infty$. Moreover, $ u_\infty - \sup_M u_\infty$ solves 
\begin{equation}
	\left\{
	\begin{aligned}
		& \chi^n_u = c \sum^n_{\a = 1} b_\a \chi^{n - \a}_u \wedge \omega^\a, \\
		& \chi_u > 0,\qquad \sup_M u = 0 .
	\end{aligned}
	\right.
\end{equation}
\end{theorem}


The paper is organized as follows. 
In Section~\ref{preliminary}, we state some preliminary knowledge. 
In Section~\ref{C2}, we generalize the $C^2$ estimate of the author~\cite{Sun2013p} and prove Theorem~\ref{introduction:theorem-C2}. 
In Section~\ref{long},  time-independent estimates are derived from the $C^2$ estimate in Section~\ref{C2}. We prove the time-independent estimates for Hermitian manifolds and K\"ahler manifolds respectively. 
In Section~\ref{gradient}, we give a new direct argument for gradient estimates. 
In Section~\ref{convergence}, we briefly review the convergence of the normalized solution, and finish the proofs of Theorem~\ref{introduction:theorem-hermitian} and Theorem~\ref{introduction:theorem-kahler}. 
In Section~\ref{J}, we study another parabolic flow used by  Collins and Sz\'ekelyhidi~\cite{CollinsSzekelyhidi2014a} and prove 
Theorem~\ref{introduction:theorem-kahler-J}.
We also discuss the lower bound and properness of the more general $J$-functional in the previous sections. It is pointed out by Gabor Sz\'ekelyhidi that this is used in their paper~\cite{CollinsSzekelyhidi2014a}.

\bigskip

\section{Preliminary}
\label{preliminary}
\setcounter{equation}{0}
\medskip

\subsection{Notations}

We denote by $\nabla$ the Chern connection of $g$. 
As in \cite{GSun12}, we express
\begin{equation}
\label{int:definition-X}
X := \chi_u \,,
\end{equation}
and thus in local coordinates
\begin{equation}
\label{int:definition-X-coefficients}
X_{i\bar j} = \chi_{i\bar j} + \bpartial_j \p_i u\,.
\end{equation}
Also, we denote the coefficients of $X^{-1}$ by $X^{i\bar j}$.

Let $S_\a (\lambda)$ denote the $\a$-th elementary symmetric polynomial of $\lambda \in \bfR^n$,
\begin{equation}
	S_\a (\boldsymbol{\lambda}) = \sum_{1 \leq i_1 < \cdots < i_\a \leq n} \lambda_{i_1} \cdots \lambda_{i_\a} \,.
\end{equation}
For a nonsingular square matrix $A$, we define $S_\a (A) = S_\a (\boldsymbol{\lambda}(A))$ where $\boldsymbol{\lambda}(A)$ denote the eigenvalues of $A$. Further, write $S_\a (X) = S_\a (\boldsymbol{\lambda}_* (X))$ and $S_\a (X^{-1}) = S_\a (\boldsymbol{\lambda}^* (X^{-1}))$ where $\boldsymbol{\lambda}_* (A)$ and $\boldsymbol{\lambda}^* (A)$ denote the eigenvalues of a Hermitian matrix $A$ with respect to $\{g_{i\bar j}\}$ and to $\{g^{i\bar j}\}$, respectively. In this paper, we shall use $S_\alpha$ to denote $S_\alpha(X^{-1})$. In local coordinates, equation~\eqref{introduction:parabolic-flow-equation} can be written in the form
\begin{equation}
\label{introduction:parabolic-equation-equivalent}
	\frac{\p u }{\p t} = \ln \frac{S_n ( X )}{\sum^n_{\a = 1} c_\a S_{n - \a}(X)} - \ln \psi ,
\end{equation}
where
\begin{equation}
	c_\a = \frac{b_\a (n - \a )! \a !}{n!} = \frac{b_\a}{C^\a_n}.
\end{equation}

\medskip
\subsection{Concavity}

Concavity (or convexity) is an important assumption in the theory of fully nonlinear ellptic and parabolic equations, e.g. Evans-Krylov theorem. Indeed, the concavity itself will play a key role in the convergence later.

By the work of Caffarelli, Nirenberg and Spruck~\cite{CNS3}, we only need to prove that $\ln \frac{S_n ( \boldsymbol{\lambda} )}{\sum^n_{\a = 1} c_\a S_{n - \a}( \boldsymbol{\lambda} )} $ is concave with respect to $\boldsymbol{\lambda}$ instead of Hermitian matrix. For convenience, rewrite
\begin{equation}
\label{preliminary:elliptic-part-1}
	\ln \frac{S_n ( \boldsymbol{\lambda} )}{\sum^n_{\a = 1} c_\a S_{n - \a}( \boldsymbol{\lambda} )} = \ln S_n ( \boldsymbol{\lambda} ) - \ln \sum^n_{\a = 1} c_\a S_{n - \a} ( \boldsymbol{\lambda} ) .
\end{equation}
Define
\begin{equation}
\label{preliminary:elliptic-hessian-matrix-term}
	B_{ij} = \frac{\p^2}{\p\lambda_j\lambda_i}\Big(\ln S_n ( \boldsymbol{\lambda} ) - \ln \sum^n_{\a = 1} c_\a S_{n - \a} ( \boldsymbol{\lambda} ) \Big) .
\end{equation}
We need to show that the matrix $\{B_{ij}\}_{n\times n}$ is non-positive definite. 

Differentiating $\ln S_n ( \boldsymbol{\lambda} ) - \ln \sum^n_{\a = 1} c_\a S_{n - \a} ( \boldsymbol{\lambda} )$ twice,
\begin{equation}
\label{preliminary-elliptic-part-1st-derivative}
	\frac{\p}{\p \lambda_i} \Big(\ln S_n ( \boldsymbol{\lambda} ) - \ln \sum^n_{\a = 1} c_\a S_{n - \a} ( \boldsymbol{\lambda} )\Big) = \frac{S_{n - 1;i}}{S_n} - \frac{\sum^{n - 1}_{\a = 1} c_\a S_{n - \a - 1;i} ( \boldsymbol{ \lambda} ) }{ \sum^n_{\a = 1} c_\a S_{n - \a} ( \boldsymbol{ \lambda} ) } ,
\end{equation}
and
\begin{equation}
\label{preliminary-elliptic-part-2nd-derivative}
\begin{aligned}
	&\, \frac{\p^2}{\p \lambda_j\p \lambda_i} \Big(\ln S_n ( \boldsymbol{\lambda} ) - \ln \sum^n_{\a = 1} c_\a S_{n - \a} ( \boldsymbol{\lambda} )\Big) \\
	=&\, \frac{S_{n - 2;ij} ( \boldsymbol{\lambda} )}{S_n ( \boldsymbol{\lambda} )} - \frac{S_{n - 1;i} ( \boldsymbol{\lambda} ) S_{n - 1;j} ( \boldsymbol{\lambda} )}{S^2_n ( \boldsymbol{\lambda} )} - \frac{\sum^{n - 2}_{\a = 1} c_\a S_{n - \a - 2;ij} ( \boldsymbol{\lambda} )}{\sum^n_{\a = 1} c_\a S_{n - \a} ( \boldsymbol{\lambda} )} \\
	&\,\qquad + \frac{\sum^{n - 1}_{\a = 1} c_\a S_{n - \a - 1;i} ( \boldsymbol{\lambda} ) \sum^{n - 1}_{\a = 1} c_\a S_{n - \a - 1;j} ( \boldsymbol{\lambda} )}{\Big(\sum^n_{\a = 1} c_\a S_{n - \a} ( \boldsymbol{\lambda} )\Big)^2} .
\end{aligned}
\end{equation}

It is well known that $\ln S_n ( \boldsymbol{\lambda} )- \ln S_{n - \a}( \boldsymbol{\lambda} )$ is concave, and consequently
\begin{equation}
\begin{aligned}
	&\, \Bigg\{\frac{S_{n - 2;ij} ( \boldsymbol{\lambda} )}{S_n ( \boldsymbol{\lambda} )} - \frac{S_{n - 1;i} ( \boldsymbol{\lambda} ) S_{n - 1;j} ( \boldsymbol{\lambda} )}{S^2_n ( \boldsymbol{\lambda} )}\Bigg\}_{n \times n} \\
	\leq&\, \Bigg\{ \frac{S_{n - \a - 2;ij} ( \boldsymbol{\lambda} )}{S_{n - \a} ( \boldsymbol{\lambda} )} - \frac{S_{n - \a - 1;i}( \boldsymbol{\lambda} ) S_{n - \a - 1;j} ( \boldsymbol{\lambda} )}{S^2_{n - \a}( \boldsymbol{\lambda} ) }\Bigg\}_{n \times n} .
\end{aligned}
\end{equation}
When $\a = n$
\begin{equation}
	\Bigg\{\frac{S_{n - 2;ij} ( \boldsymbol{\lambda} )}{S_n ( \boldsymbol{\lambda} )} - \frac{S_{n - 1;i} ( \boldsymbol{\lambda} ) S_{n - 1;j} ( \boldsymbol{\lambda} )}{S^2_n ( \boldsymbol{\lambda} )}\Bigg\}_{n \times n} \leq \boldsymbol{0} ,
\end{equation}
and when $\a = n - 1$
\begin{equation}
	\Bigg\{\frac{S_{n - 2;ij} ( \boldsymbol{\lambda} )}{S_n ( \boldsymbol{\lambda} )} - \frac{S_{n - 1;i} ( \boldsymbol{\lambda} ) S_{n - 1;j} ( \boldsymbol{\lambda} )}{S^2_n ( \boldsymbol{\lambda} )}  \Bigg\}_{n \times n} \leq  \Bigg\{ - \frac{1}{S^2_1 ( \boldsymbol{\lambda} )} \Bigg\}_{n \times n} .
\end{equation}
Hence,
\begin{equation}
\label{preliminary-elliptic-part-jacobian-1}
\begin{aligned}
	\{ B_{ij} \}_{n \times n} 
	\leq& \Bigg\{ \sum^n_{\a = 1} \frac{c_a S_{ n - \a}}{\sum^n_{\a = 1} c_a S_{n - \a}}\Bigg( \frac{S_{n - \a - 2;ij}}{S_{n - \a}} - \frac{S_{n - \a - 1;i} S_{n - \a - 1;j}}{S^2_{n - \a}} \Bigg) \\
	&\;\; - \frac{\sum^{n - 2}_{\a = 1} c_\a S_{n - \a - 2;ij}}{\sum^n_{\a = 1} c_\a S_{n - \a}} + \frac{\sum^{n - 1}_{\a = 1} c_\a S_{n - \a - 1;i} \sum^{n - 1}_{\a = 1} c_\a S_{n - \a - 1;j}}{\Big(\sum^n_{\a = 1} c_\a S_{n - \a}\Big)^2} \Bigg\}_{n \times n} \\
	=& \Bigg\{ - \sum^n_{\a = 1} \frac{c_a S_{ n - \a}}{\sum^n_{\a = 1} c_a S_{n - \a}} \frac{S_{n - \a - 1;i} }{S_{n - \a}} \frac{ S_{n - \a - 1;j}}{S_{n - \a}}  \\
	&\;\;\; + \sum^{n - 1}_{\a = 1} \frac{c_\a S_{n - \a}}{\sum^n_{\a = 1} c_a S_{n - \a}} \frac{S_{n - \a - 1;i}}{S_{n - \a}} \sum^{n - 1}_{\a = 1} \frac{c_\a S_{n - \a}}{\sum^n_{\a = 1} c_a S_{n - \a}} \frac{S_{n - \a - 1;j}}{S_{n - \a}}\Bigg\}_{n \times n}
\end{aligned} 
\end{equation}
For any real vector
\(
	\boldsymbol{\eta}  = (\eta^1, \cdots , \eta^n) ,
\)
\begin{equation}
\label{preliminary-elliptic-part-jacobian-2}
\begin{aligned}
	\sum_{i,j} B_{ij} \eta^i \eta^j
	\leq&\, - \sum^n_{\a = 1} \frac{c_a S_{ n - \a}}{\sum^n_{\a = 1} c_a S_{n - \a}} \Bigg| \sum_i \frac{S_{n - \a - 1;i} \eta^i }{S_{n - \a}} \Bigg|^2 \\
	&\;\; + \Bigg| \sum^n_{\a = 1} \frac{c_\a S_{n - \a}}{\sum^n_{\a = 1} c_a S_{n - \a}} \sum_i \frac{S_{n - \a - 1;i} \eta^i}{S_{n - \a}} \Bigg|^2 ,
\end{aligned}
\end{equation}
where $S_{-1 ;i} = 0$ by convention. Applying H\"older inequality,
\begin{equation}
\label{preliminary-elliptic-part-jacobian-3}
	\Bigg| \sum^n_{\a = 1} \frac{c_\a S_{n - \a}}{\sum^n_{\a = 1} c_a S_{n - \a}} \sum_i \frac{S_{n - \a - 1;i} \eta^i}{S_{n - \a}} \Bigg|^2 \leq \sum^n_{\a = 1} \frac{c_\a S_{n - \a}}{\sum^n_{\a = 1} c_a S_{n - \a}} \Bigg|\sum_i \frac{S_{n - \a - 1;i} \eta^i}{S_{n - \a}} \Bigg|^2 , 
\end{equation}
and thus
\(
	\sum_{i,j} B_{ij} \eta^i \eta^j \leq 0 .
\)

\medskip
\subsection{Formulas}
We state some useful formulas at a given point $p$. Assume that at the point $p$,  $g_{i\bar j} = \delta_{ij}$ and $X_{i\bar j}$ is diagonal in a specific chart.  
In this paper, we call such local coordinates normal coordinate charts around $p$.

We recall the formula from \cite{GL10},
\begin{equation}
	\ol{X_{i\bar jk}} = X_{j\bar i\bar k},
\end{equation}
and
\begin{equation}
\label{preliminary:formula-X-1}
\begin{aligned}
     	X_{i\bar ij\bar j} - X_{j\bar ji\bar i} &= R_{j\bar ji\bar i}X_{i\bar i}  -  R_{i\bar ij\bar j}X_{j\bar j} + 2 \mathfrak{Re} \Big\{\sum_p \overline{T^p_{ij}}X_{i\bar pj}\Big\} \\
     	&\hspace{6em} - \sum_{p} T^p_{ij} \overline{T^p_{ij}} X_{p\bar p} - G_{i\bar ij\bar j},
\end{aligned}
\end{equation}
where
\begin{equation}
\label{preliminary:formula-G-coefficient}
\begin{aligned}
    	G_{i\bar ij\bar j} &= \chi_{j\bar ji\bar i} - \chi_{i\bar ij\bar j} + \sum_p R_{j\bar ji\bar p}\chi_{p\bar i} -\sum_p R_{i\bar ij\bar p}\chi_{p\bar j} \\
    	&\hspace{3em} + 2\mathfrak{Re}\Big\{\sum_p \overline{T^p_{ij}}\chi_{i\bar pj} \Big\} - \sum_{p,q}T^p_{ij}\overline{T^q_{ij}}\chi_{p\bar q}\,.
\end{aligned}
\end{equation}

For convenience, we can rewrite equation~\eqref{introduction:parabolic-flow-equation} in the form
\begin{equation}
\label{preliminary:parabplic-equation-equivalent-1}
	\p_t u 
	= - \ln \sum^n_{\a = 1} c_\a S_\a - \ln \psi ,
\end{equation}
and condition~\eqref{introduction:cone-condition} in the form
\begin{equation}
\label{preliminary:cone-condition-equivalent-1}
	\sum^{n - 1}_{\a = 1} c_\a S_{\a} \big((\chi_v|k)^{-1}\big) < \frac{1}{\psi},
\end{equation}
for all $k$, where $(\chi_v|k)$ denotes the $(k,k)$ minor matrix of $\chi_v$.

Differentiating the equation at $p$
\begin{equation}
\label{preliminary:equation-time-derivative}
	\p_t (\p_t u) = \sum^n_{\a = 1}\frac{ c_\a}{\sum^n_{\a = 1} c_\a S_\a}  \sum_i S_{\a - 1;i} (X^{i\bar i})^2 (\p_t u)_{i\bar i} ,
\end{equation}
\begin{equation}
\label{preliminary:equation-1st-derivative}
	\p_t u_l = \sum^n_{\a = 1}\frac{ c_\a}{\sum^n_{\a = 1} c_\a S_\a}  \sum_i S_{\a - 1;i} (X^{i\bar i})^2 X_{i\bar il} - \frac{\p_l \psi}{\psi} ,
\end{equation}
and
\begin{equation}
\label{preliminary:equation-2nd-derivative}
\begin{aligned}
	\p_t u_{l\bar l} =&\,  \sum^n_{\a = 1} \frac{c_\a}{\sum^n_{\a = 1} c_\a S_\a} \Bigg[ \sum_i S_{\a - 1;i} (X^{i\bar i})^2 X_{i\bar il\bar l} - \sum_{i,j} S_{\a - 1;i} (X^{i\bar i})^2 X^{j\bar j} X_{j\bar i\bar l} X_{i\bar jl} \\
	&\,\hspace{12em} - \sum_{i,j} S_{\a - 1;i} (X^{i\bar i})^2 X^{j\bar j} X_{j\bar il} X_{i\bar j\bar l} \\
	&\,\hspace{12em} + \sum_{i\neq j} S_{\a -2;ij} (X^{i\bar i})^2 (X^{j\bar j})^2 X_{j\bar il} X_{i\bar j\bar l} \\
	&\,\hspace{12em} - \sum_{i\neq j} S_{\a - 2;ij} (X^{i\bar i})^2 (X^{j\bar j})^2 X_{j\bar jl} X_{i\bar i\bar l}\Bigg] \\
	&\, + \Bigg|\sum^n_{\a = 1}\frac{ c_\a}{\sum^n_{\a = 1} c_\a S_\a} \sum_i S_{\a - 1;i} (X^{i\bar i})^2 X_{i\bar il} \Bigg|^2 - \frac{\bpartial_l\p_l \psi}{\psi} + \frac{\p_l \psi \bpartial_l \psi}{\psi} .
\end{aligned}
\end{equation}
Recall that the fundamental polynomials have strong concavity in $\Gamma_n$, which was shown in \cite{GLZ, FLM11}.
\begin{theorem}
\label{preliminary:theorem-concavity}
    For $\lambda\in \Gamma_n$, $\xi = (\xi_1,\cdots\xi_n) \in \bfC^n$, We have
    \begin{equation}
        \sum^n_{i=1} \frac{S_{\a -1;i}(\lambda)}{\lambda_i} \xi_i\bar\xi_i + \sum_{i,j} S_{\a -2;ij}(\lambda) \xi_i\bar\xi_j \geq \sum_{i,j} \frac{S_{\a -1;i}(\lambda)S_{\a -1;j}(\lambda)}{S_\a (\lambda)}\xi_i\bar\xi_j \geq 0 \,.
    \end{equation}
\end{theorem}
So we can control some terms in the right of \eqref{preliminary:equation-2nd-derivative},
\begin{equation}
\label{preliminary:equation-2nd-derivative-1}
\begin{aligned}
	\p_t X_{l\bar l} 
	&\leq 
	\sum^n_{\a = 1} \frac{c_\a}{\sum^n_{\a = 1} c_\a S_\a} \Bigg[ \sum_i S_{\a - 1;i} (X^{i\bar i})^2 X_{i\bar il\bar l} - \sum_{i,j} S_{\a - 1;i} (X^{i\bar i})^2 X^{j\bar j} X_{j\bar i\bar l} X_{i\bar jl} \Bigg] \\
	&\hspace{5em} - \sum^n_{\a = 1} \frac{c_\a S_\a}{\sum^n_{\a = 1} c_\a S_\a} \Big|\p_l (\ln S_\a) \Big|^2 + \Bigg|\sum^n_{\a = 1} \frac{c_\a S_\a}{\sum^n_{\a = 1} c_\a S_\a} \p_l (\ln S_\a) \Bigg|^2 \\
	&\hspace{12em} - \frac{\bpartial_l\p_l \psi}{\psi} + \frac{\p_l \psi \bpartial_l \psi}{\psi} . 
\end{aligned}
\end{equation}
Applying H\"older inequality to \eqref{preliminary:equation-2nd-derivative-1}, and then summing it over $l$,
\begin{equation}
\label{preliminary:equation-2nd-derivative-2}
\begin{aligned}
	\p_t w \leq C_1 + \sum^n_{\a = 1} \frac{c_\a}{\sum^n_{\a = 1} c_\a S_\a} \Bigg[ &\sum_{i,l} S_{\a - 1;i} (X^{i\bar i})^2 X_{i\bar il\bar l} \\
	&\; - \sum_{i,j,l} S_{\a - 1;i} (X^{i\bar i})^2 X^{j\bar j} X_{j\bar i\bar l} X_{i\bar jl} \Bigg] .
\end{aligned}
\end{equation}

Since $v \in C^2(M)$ and $\chi_v > 0$, we may assume 
\begin{equation}
	\epsilon \omega \leq \chi_v \leq \epsilon^{-1} \omega
\end{equation}
for some $\epsilon > 0$.

Morevoer, by the maximum principle, $\p_t u$ attains its extremal values at $t = 0$,
\begin{equation}
\label{preliminary:maximum-1}
	\inf_{M \times \{0\} } \p_t u\leq \p_t u \leq \sup_{M \times \{0\}} \p_t u .
\end{equation}
It follows immediately that
\begin{equation}
\label{preliminary:maximum-2}
	|\p_t \tilde u| \leq \sup_{M \times \{0\}} \p_t u - \inf_{M \times \{0\}} \p_t u ,
\end{equation}
and 
\begin{equation}
\label{preliminary:maximum-3}
	 \inf_{M } \frac{S_n ( \chi )}{\psi\sum^n_{\a = 1} c_\a S_{n - \a}(\chi)} \leq  \frac{S_n ( X )}{ \psi\sum^n_{\a = 1} c_\a S_{n - \a}(X)} \leq \sup_{M }  \frac{S_n ( \chi )}{\psi\sum^n_{\a = 1} c_\a S_{n - \a}(\chi)}  .
\end{equation}
Therefore, the flow remains Hermitian at any time.

\bigskip

\section{The second order estimate}
\label{C2}
\setcounter{equation}{0}
\medskip

In this section we derive the partial second order estimate for admissible solutions.

In order to prove the second order estimate, we need a key inequality. Guan~\cite{Guan2014a}, Collins and Sz\'ekelyhidi~\cite{CollinsSzekelyhidi2014a} and Sz\'ekelyhidi~\cite{Szekelyhidi2014b} have more general statements for elliptic equations.

\begin{lemma}
\label{lemma-inequality}
There are constants $N , \theta > 0$ such that when $w \geq N$ at a point $p$,
\begin{equation}
\begin{aligned}
	&\, \sum^n_{\a = 1} \frac{c_\a}{\sum^n_{\a = 1} c_\a S_\a} \sum_i S_{\a - 1;i} (X^{i\bar i})^2 (u_{i\bar i} - v_{i\bar i}) \\
	\leq&\, - \ln \Big(\psi\sum^n_{\a = 1} c_\a S_\a \Big) - \theta - \theta \sum^n_{\a = 1} \frac{c_\a}{\sum^n_{\a = 1} c_\a S_\a} \sum_i S_{\a - 1;i} (X^{i\bar i})^2 .
\end{aligned}
\end{equation}
\end{lemma}
\begin{proof}
Without loss of generality, we may assume that $X_{1\bar 1} \geq \cdots \geq X_{n\bar n}$. Direct calculation shows that
\begin{equation}
\begin{aligned}
	&\, \sum^n_{\a = 1} \frac{c_\a}{\sum^n_{\a = 1} c_\a S_\a} \sum_i S_{\a - 1;i} (X^{i\bar i})^2 (u_{i\bar i} - v_{i\bar i}) \\
	\leq&\, \sum^n_{\a = 1} \frac{c_\a}{\sum^n_{\a = 1} c_\a S_\a} \sum_i S_{\a - 1;i} X^{i\bar i} - \epsilon \sum^n_{\a = 1} \frac{c_\a}{\sum^n_{\a = 1} c_\a S_\a} \sum_i S_{\a - 1;i} (X^{i\bar i})^2 \\
	=&\, \Big(1 - \frac{\epsilon }{2n} S_1\Big) \sum^n_{\a = 1} \frac{c_\a \a S_\a}{\sum^n_{\a = 1} c_\a S_\a} - \frac{\epsilon}{2} \sum^n_{\a = 1} \frac{c_\a}{\sum^n_{\a = 1} c_\a S_\a} \sum_i S_{\a - 1;i} (X^{i\bar i})^2 .
\end{aligned}
\end{equation}
Noting that for all $s > 0$
\begin{equation}
	\sup_M \psi s - \ln s \geq 1 + \ln \psi,
\end{equation}
and thus if $S_1 \geq \frac{2n}{\epsilon} (1 + \sup_M \psi)$, we have
\begin{equation}
\begin{aligned}
	&\, \sum^n_{\a = 1} \frac{c_\a}{\sum^n_{\a = 1} c_\a S_\a} \sum_i S_{\a - 1;i} (X^{i\bar i})^2 (u_{i\bar i} - v_{i\bar i}) \\
	\leq&\, - \ln \sum^n_{\a = 1} c_\a S_\a - \ln \psi - 1  - \frac{\epsilon}{2} \sum^n_{\a = 1} \frac{c_\a}{\sum^n_{\a = 1} c_\a S_\a} \sum_i S_{\a - 1;i} (X^{i\bar i})^2 .
\end{aligned}
\end{equation}

Now we just need to consider the case $X_{1\bar 1} \geq \cdots \geq X_{n\bar n} \geq \frac{\epsilon}{2n (1 + \sup_M \psi)}$. Rewriting
\begin{equation}
\begin{aligned}
	&\, \sum^n_{\a = 1} \frac{c_\a}{\sum^n_{\a = 1} c_\a S_\a}  \sum_i S_{\a - 1;i} (X^{i\bar i})^2 (u_{i\bar i} - v_{i\bar i}) \\
	\leq&\, \sum^n_{\a = 1} \frac{c_\a}{\sum^n_{\a = 1} c_\a S_\a}  \sum^n_{i = 2} S_{\a - 1;i} (X^{i\bar i})^2 (X_{i\bar i} - (\chi_{i\bar i} + v_{i\bar i} - \delta)) \\
	&\, - \delta  \sum^n_{\a = 1} \frac{c_\a}{\sum^n_{\a = 1} c_\a S_\a}  \sum_i S_{\a - 1;i} (X^{i\bar i})^2 + \sum^n_{\a = 1} \frac{c_\a}{\sum^n_{\a = 1} c_\a S_\a}  S_{\a - 1;1} X^{1\bar 1} .
\end{aligned}
\end{equation}
If $\delta > 0$ is small enough, $\chi - \delta g$ is still positive definite and satisfies \eqref{preliminary:cone-condition-equivalent-1}. So there are $N' > 0$ and $\sigma > 0$ such that for all $\lambda > N'$
\begin{equation}
	\ln \sum^n_{\a = 1} c_\a S_\a (\chi'^{-1}) <- \ln{\psi} - \sigma ,
\end{equation}
where
\begin{equation}
	\chi' = \left\{ 
	\begin{aligned}
	& \lambda 	& \\
	&			& (\chi_v - \delta g|1)
	\end{aligned}
	\right\}_{n\times n} .
\end{equation}

Since $- \ln \sum^n_{\a = 1} c_\a S_\a$ is concave,
\begin{equation}
\begin{aligned}
	&\, \sum^n_{\a = 1} \frac{c_\a}{\sum^n_{\a = 1} c_\a S_\a}  \sum_i S_{\a - 1;i} (X^{i\bar i})^2 (u_{i\bar i} - v_{i\bar i}) \\
	\leq&\, - \ln \sum^n_{\a = 1} c_\a S_\a + \ln \sum^n_{\a = 1} c_\a S_\a (\chi''^{-1}) - \delta  \sum^n_{\a = 1} \frac{c_\a}{\sum^n_{\a = 1} c_\a S_\a}  \sum_i S_{\a - 1;i} (X^{i\bar i})^2 \\
	&\, + \sum^n_{\a = 1} \frac{c_\a}{\sum^n_{\a = 1} c_\a S_\a}  S_{\a - 1;1} X^{1\bar 1},
\end{aligned}
\end{equation}
where
\begin{equation}
	\chi'' = \left\{ 
	\begin{aligned}
	& X_{1\bar 1} & \\
	&			& (\chi_v - \delta g|1)
	\end{aligned}
	\right\}_{n\times n} .
\end{equation}
If $X_{1\bar 1} > N'$,
\begin{equation}
\begin{aligned}
	&\, \sum^n_{\a = 1} \frac{c_\a}{\sum^n_{\a = 1} c_\a S_\a}  \sum_i S_{\a - 1;i} (X^{i\bar i})^2 (u_{i\bar i} - v_{i\bar i}) \\
	<&\,  - \ln \sum^n_{\a = 1} c_\a S_\a - \ln \psi -\sigma - \delta  \sum^n_{\a = 1} \frac{c_\a}{\sum^n_{\a = 1} c_\a S_\a}   \sum_i S_{\a - 1;i} (X^{i\bar i})^2 \\
	&\, + \sum^n_{\a = 1} \frac{c_\a}{\sum^n_{\a = 1} c_\a S_\a}   S_{\a - 1;1} X^{1\bar 1} \\
	\leq&\, - \ln \sum^n_{\a = 1} c_\a S_\a - \ln \psi -\sigma - \delta  \sum^n_{\a = 1} \frac{c_\a}{\sum^n_{\a = 1} c_\a S_\a}   \sum_i S_{\a - 1;i} (X^{i\bar i})^2  \\
	&\; + \sum^n_{\a = 1} \frac{c_\a}{\sum^n_{\a = 1} c_\a S_\a}  C^{\a - 1}_{n - 1} \Big(\frac{2n}{\epsilon} (1 + \sup_M \psi)\Big)^{\a - 1} X^{1\bar 1}.
\end{aligned}
\end{equation}
Using the bound~\eqref{preliminary:maximum-3} and let $X_{1\bar 1}$ be sufficiently large, we have
\begin{equation}
\begin{aligned}
	&\,\sum^n_{\a = 1} \frac{c_\a}{\sum^n_{\a = 1} c_\a S_\a}  \sum_i S_{\a - 1;i} (X^{i\bar i})^2 (u_{i\bar i} - v_{i\bar i}) \\
	\leq&\,  - \ln \sum^n_{\a = 1} c_\a S_\a - \ln \psi - \frac{\sigma}{2} - \delta  \sum^n_{\a = 1} \frac{c_\a}{\sum^n_{\a = 1} c_\a S_\a}   \sum_i S_{\a - 1;i} (X^{i\bar i})^2  .
\end{aligned}
\end{equation}

\end{proof}

\begin{theorem}
\label{C2:theorem-C2-estimate}
Let $u \in C^4 (M\times [0, T))$ be an admissible solution to equation~\eqref{introduction:parabolic-flow-equation} and $w = \Delta u + tr\chi$. Then there are uniform constants $C$ and $A$ such that
\begin{equation}
	w \leq C e^{A (u - \inf_{M\times [ 0, t]} u)},
\end{equation}
where $C$, $A$ depend only on geometric data.
\end{theorem}

\begin{proof}
We consider the function $w e^\phi$ where $\phi$ is to specified later. Suppose that $\ln w e^\phi$ achieves its maximum at some point $(p,t_0) \in M_t = M \times (0,t]$. Choose a local chart around $p$ such that $g_{i\bar j} = \delta_{ij}$ and $X_{i\bar j}$ is diagonal at $p$ when $t = t_0$. Therefore, we have at the point $(p,t_0)$,
\begin{equation}
\label{C2:test-derivative-1}
	\frac{\p_l w}{w} + \p_l \phi = 0 ,
\end{equation}
\begin{equation}
\label{C2:test-derivative-2}
	\frac{\bpartial_l w}{w} + \bpartial_l \phi = 0 ,
\end{equation}
\begin{equation}
\label{C2:test-derivative-3}
	\frac{\p_t w}{w} + \p_t \phi \geq 0 ,
\end{equation}
and
\begin{equation}
\label{C2:test-derivative-4}
	\frac{\bpartial_l\p_l w}{w} - \frac{|\p_l w |^2 }{w^2} + \bpartial_l\p_l \phi \leq 0 .
\end{equation}
Without loss of generality, we may assume that $w \gg 1$. Otherwise, the proof is finished.

Since
\begin{equation}
\label{C2:square}
\begin{aligned}
	&\, \sum_{i,j,l} S_{\a - 1;i} (X^{i\bar i})^2 X^{j\bar j} \Big|X_{i\bar jl} - X_{j\bar l} \frac{\p_i w}{w} - T^j_{il} X_{j\bar j}\Big|^2 \\
	=&\, \sum_{i,j,l} S_{\a - 1;i} (X^{i\bar i})^2 X^{j\bar j} X_{i\bar jl} X_{j\bar i\bar l} - \sum_i S_{\a - 1;i} (X^{i\bar i})^2 \frac{|\p_i w|^2}{w} \\
	&\, + \sum_{i,j,l} S_{\a - 1;i} (X^{i\bar i})^2 X_{j\bar j} T^j_{il} \ol{T^j_{il}} + \frac{2}{w} \sum_{i,j} S_{\a - 1;i} (X^{i\bar i})^2 \mathfrak{Re} \Big\{\sum_k \hat{T}^k_{ij} \chi_{k\bar j} {\bpartial_i w} \Big\} \\
	&\, - 2 \sum_{i,j,l} S_{\a - 1;i} (X^{i\bar i})^2 \mathfrak{Re}\Big\{X_{i\bar jl} \ol{T^j_{il}}\Big\},
\end{aligned}
\end{equation}
where $\hat{T}$ denotes the torsion with respect to the Hermitian metric $\chi$. Then
\begin{equation}
\label{C2:square-1}
\begin{aligned}
	 &  - \frac{2}{w} \sum_{i,j} S_{\a - 1;i} (X^{i\bar i})^2 \mathfrak{Re} \Big\{\sum_k \hat{T}^k_{ij} \chi_{k\bar j} {\bpartial_i w} \Big\} + \sum_i S_{\a - 1;i} (X^{i\bar i})^2 \frac{|\p_i w|^2}{w}\\
	 &\leq \sum_{i,j,l} S_{\a - 1;i} (X^{i\bar i})^2 X^{j\bar j} X_{i\bar jl} X_{j\bar i\bar l} - 2 \sum_{i,j,l} S_{\a - 1;i} (X^{i\bar i})^2 \mathfrak{Re}\Big\{X_{i\bar jl} \ol{T^j_{il}} \Big\} \\
	 &\,\hspace{12em} + \sum_{i,j,l} S_{\a - 1;i} (X^{i\bar i})^2 X_{j\bar j} T^j_{il} \ol{T^j_{il}} .
\end{aligned}
\end{equation}
So by \eqref{preliminary:formula-X-1} and \eqref{C2:square-1},
\begin{equation}
\label{C2:inequality-1}
\begin{aligned}
	&\, \sum_{i,j,l} S_{\a - 1;i} (X^{i\bar i})^2 X^{j\bar j} X_{j\bar i\bar l} X_{i\bar jl} - \sum_{i,l} S_{\a - 1;i} (X^{i\bar i})^2 X_{i\bar il\bar l} \\
	\geq&\, - \frac{2}{w} \sum_{i,j} S_{\a - 1;i} (X^{i\bar i})^2 \mathfrak{Re} \Big\{\sum_k \hat{T}^k_{ij} \chi_{k\bar j} {\bpartial_i w} \Big\} + \sum_i S_{\a - 1;i} (X^{i\bar i})^2 \frac{|\p_i w|^2}{w} \\
	&\,  - \sum_i S_{\a - 1;i} (X^{i\bar i})^2 \bpartial_i \p_i w  + \sum_{i,l} S_{\a - 1;i} (X^{i\bar i})^2 \Big(- R_{l\bar li\bar i} X_{i\bar i} + R_{i\bar il\bar l} X_{l\bar l} + G_{i\bar il\bar l}\Big) .
\end{aligned}
\end{equation}
Substituting \eqref{C2:test-derivative-4} into \eqref{C2:inequality-1}
\begin{equation}
\label{C2:inequality-2}
\begin{aligned}
	&\, \sum_{i,j,l} S_{\a - 1;i} (X^{i\bar i})^2 X^{j\bar j} X_{j\bar i\bar l} X_{i\bar jl} - \sum_{i,l} S_{\a - 1;i} (X^{i\bar i})^2 X_{i\bar il\bar l} \\
	\geq&\, - \frac{2}{w} \sum_{i,j} S_{\a - 1;i} (X^{i\bar i})^2 \mathfrak{Re} \{\sum_k \hat{T}^k_{ij} \chi_{k\bar j} {\bpartial_i w} \} + w \sum_i S_{\a - 1;i} (X^{i\bar i})^2 \bpartial_i \p_i \phi\\
	&\;\; + \sum_{i,l} S_{\a - 1;i} (X^{i\bar i})^2 \Big(- R_{l\bar li\bar i} X_{i\bar i} + R_{i\bar il\bar l} X_{l\bar l} + G_{i\bar il\bar l}\Big) .
\end{aligned}
\end{equation}
Combining \eqref{preliminary:equation-2nd-derivative-2} and \eqref{C2:inequality-2},
\begin{equation}
\label{C2:inequality-3}
\begin{aligned}
	\p_t w	\leq&\, C_1 + \sum^n_{\a = 1} \frac{c_\a}{\sum^n_{\a = 1} c_\a S_\a} \Bigg[ \frac{2}{w} \sum_{i,j} S_{\a - 1;i} (X^{i\bar i})^2 \mathfrak{Re} \Big\{\sum_k \hat{T}^k_{ij} \chi_{k\bar j} {\bpartial_i w} \Big\} \\
	&\hspace{12em} - w \sum_i S_{\a - 1;i} (X^{i\bar i})^2 \bpartial_i \p_i \phi \Bigg]\\
	&\, - \sum^n_{\a = 1} \frac{c_\a}{\sum^n_{\a = 1} c_\a S_\a} \Bigg[\sum_{i,l} S_{\a - 1;i} (X^{i\bar i})^2 \Big(- R_{l\bar li\bar i} X_{i\bar i} + R_{i\bar il\bar l} X_{l\bar l} + G_{i\bar il\bar l}\Big)\Bigg]  .
\end{aligned}
\end{equation}
Then
\begin{equation}
\label{C2:inequality-4}
\begin{aligned}
	\p_t w	\leq&\, \sum^n_{\a = 1} \frac{c_\a}{\sum^n_{\a = 1} c_\a S_\a} \Bigg[ \frac{2}{w} \sum_{i,j} S_{\a - 1;i} (X^{i\bar i})^2 \mathfrak{Re} \Big\{\sum_k \hat{T}^k_{ij} \chi_{k\bar j} {\bpartial_i w} \Big\} \\
	&\hspace{10em} - w \sum_i S_{\a - 1;i} (X^{i\bar i})^2 \bpartial_i \p_i \phi \Bigg]\\
	&\,  + C_3 w \sum^n_{\a = 1} \frac{c_\a}{\sum^n_{\a = 1} c_\a S_\a} \sum_{i} S_{\a - 1;i} (X^{i\bar i})^2  + C_2 .
\end{aligned}
\end{equation}
By \eqref{C2:test-derivative-2} and \eqref{C2:test-derivative-3},
\begin{equation}
\label{C2:inequality-5}
\begin{aligned}
	w \p_t\phi \geq&\,  \sum^n_{\a = 1} \frac{c_\a}{\sum^n_{\a = 1} c_\a S_\a} \Bigg[ 2 \sum_{i,j} S_{\a - 1;i} (X^{i\bar i})^2 \mathfrak{Re} \Big\{\sum_k \hat{T}^k_{ij} \chi_{k\bar j} \bpartial_i \phi \Big\}\\
	&\hspace{10em}  + w \sum_i S_{\a - 1;i} (X^{i\bar i})^2 \bpartial_i \p_i \phi \Bigg]\\
	&\, - C_3 w \sum^n_{\a = 1} \frac{c_\a}{\sum^n_{\a = 1} c_\a S_\a} \sum_{i} S_{\a - 1;i} (X^{i\bar i})^2  - C_2 .
\end{aligned}
\end{equation}

To apply a trick due to Phong and Sturm~\cite{PhongSturm10}, we specify $\phi$ as follows,
\begin{equation}
\label{C2:definition-phi}
	\phi := - A (u - v) + \frac{1}{u - v - \inf_{M_t} (u - v) + 1} = - A (u - v) + E_1 .
\end{equation}
We may assume $A > N \gg 1$, where $N$ is the crucial constant in Lemma~\ref{lemma-inequality}.

It is easy to see that
\begin{equation}
\label{C2:phi-derivative-1}
	\p_t \phi = - (A + E^2_1) \p_t u ,
\end{equation}
\begin{equation}
\label{C2:phi-derivative-2}
	\p_i \phi = - (A + E^2_1) (u_i - v_i) ,
\end{equation}
and
\begin{equation}
\label{C2:phi-derivative-3}
	\bpartial_i \p_i \phi = - (A + E^2_1) (u_{i\bar i} - v_{i\bar i}) + 2 | u_i - v_i |^2 E^3_1 .
\end{equation}
By \eqref{C2:phi-derivative-2} and \eqref{C2:phi-derivative-3}, 
\begin{equation}
\label{C2:inequality-6}
\begin{aligned}
	&\, 2 \sum_{i,j} S_{\a - 1;i} (X^{i\bar i})^2 \mathfrak{Re} \Big\{\sum_k \hat{T}^k_{ij} \chi_{k\bar j} \bpartial_i \phi\Big\} + w \sum_i S_{\a - 1;i} (X^{i\bar i})^2 \bpartial_i \p_i \phi \\
	=&\, - 2 (A + E^2_1) \sum_i S_{\a - 1;i} (X^{i\bar i})^2 \mathfrak{Re} \Big\{\sum_{j,k} \hat{T}^k_{ij} \chi_{k\bar j} (u_{\bar i} - v_{\bar i})\Big\} \\
	&\, +2 E^3_1 w \sum_i S_{\a - 1;i} (X^{i\bar i})^2 |u_i - v_i| - (A + E^2_1) w \sum_i S_{\a - 1;i} (X^{i\bar i})^2 (u_{i\bar i} - v_{i\bar i}) .
\end{aligned}
\end{equation}
Using Schwarz inequality to control the first term in the right,
\begin{equation}
\label{C2:inequality-7}
\begin{aligned}
	&\, 2 (A + E^2_1) \Bigg| \sum_i S_{\a - 1;i} (X^{i\bar i})^2 \mathfrak{Re} \Big\{\sum_{j,k} \hat{T}^k_{ij} \chi_{k\bar j} (u_{\bar i} - v_{\bar i})\Big\}\Bigg| \\
	\leq&\, w E^3_1 \sum_i S_{\a - 1;i} (X^{i\bar i})^2 |u_i - v_i|^2 + \frac{C_4 A^2}{w E^3_1} \sum_i S_{\a - 1;i} (X^{i\bar i})^2 ,
\end{aligned}
\end{equation}
and hence
\begin{equation}
\label{C2:inequality-8}
\begin{aligned}
	&\, 2 \sum_{i,j} S_{\a - 1;i} (X^{i\bar i})^2 \mathfrak{Re} \Big\{\sum_k \hat{T}^k_{ij} \chi_{k\bar j} \bpartial_i \phi\Big\} + w \sum_i S_{\a - 1;i} (X^{i\bar i})^2 \bpartial_i \p_i \phi \\
	\geq&\, - \frac{C_4 A^2}{w E^3_1} \sum_i S_{\a - 1;i} (X^{i\bar i})^2 - (A + E^2_1) w \sum_i S_{\a - 1;i} (X^{i\bar i})^2 (u_{i\bar i} - v_{i\bar i}) .
\end{aligned}
\end{equation}
Substituting \eqref{C2:phi-derivative-1} and \eqref{C2:inequality-8} into \eqref{C2:inequality-5}
\begin{equation}
\label{C2:inequality-9}
\begin{aligned}
	 w (A + E^2_1) \p_t u \leq&\,   \frac{C_4 A^2}{w E^3_1} \sum^n_{\a = 1} \frac{c_\a}{\sum^n_{\a = 1} c_\a S_\a}   \sum_i S_{\a - 1;i} (X^{i\bar i})^2  \\
	&\, + (A + E^2_1) w \sum^n_{\a = 1} \frac{c_\a}{\sum^n_{\a = 1} c_\a S_\a}  \sum_i S_{\a - 1;i} (X^{i\bar i})^2 (u_{i\bar i} - v_{i\bar i}) \\
	&\, + C_3 w \sum^n_{\a = 1} \frac{c_\a}{\sum^n_{\a = 1} c_\a S_\a} \sum_{i} S_{\a - 1;i} (X^{i\bar i})^2  + C_2 ,
\end{aligned}
\end{equation}
that is
\begin{equation}
\label{C2:inequality-10}
\begin{aligned}
	 0 \leq&\,   w (A + E^2_1) \Bigg[ \ln \Big(\psi \sum^n_{\a = 1} c_\a S_\a\Big)  \\
	 &\hspace{5em}+ \sum^n_{\a = 1} \frac{c_\a}{\sum^n_{\a = 1} c_\a S_\a}  \sum_i S_{\a - 1;i} (X^{i\bar i})^2 (u_{i\bar i} - v_{i\bar i}) \Bigg]  \\
	&\, + \Big( \frac{C_4 A^2}{w E^3_1} + C_3 w \Big) \sum^n_{\a = 1} \frac{c_\a}{\sum^n_{\a = 1} c_\a S_\a}   \sum_i S_{\a - 1;i} (X^{i\bar i})^2  + C_2 .
\end{aligned}
\end{equation}

If $w > A (u - v - \inf_{M_t} (u - v) + 1)^{\frac{3}{2}} = A E^{-\frac{3}{2}}_1$, there is $\theta > 0$ such that,
\begin{equation}
\label{C2:inequality-11}
\begin{aligned}
	 0 
	\leq&\, - \theta w (A + E^2_1) \Bigg[ 1 + \sum^n_{\a = 1} \frac{c_\a}{\sum^n_{\a = 1} c_\a S_\a} \sum_i S_{\a - 1;i} (X^{i\bar i})^2 \Bigg]  \\
	&\, + w (C_3 + C_4) \sum^n_{\a = 1} \frac{c_\a}{\sum^n_{\a = 1} c_\a S_\a}   \sum_i S_{\a - 1;i} (X^{i\bar i})^2  + C_2 .
\end{aligned}
\end{equation}
This gives a bound $w \leq 1$ at $(p ,t_0)$ if we pick a sufficiently large $A$, which contradicts our assumption.

On the other hand, if $w \leq A (u - v - \inf_{M_t} (u - v) + 1)^{\frac{3}{2}}$,
\begin{equation}
\label{C2:inequality-12}
\begin{aligned}
	w e^\phi \leq w e^\phi |_{(p,t_0)} &\leq A (u - v - \inf_{M_t} (u - v) + 1)^{\frac{3}{2}} e^{- A(u - v) + 1} |_{(p,t_0)} \\
	&\leq A e^2 e^{- A \inf_{M_t} (u - v)} ,
\end{aligned}
\end{equation}
and hence
\begin{equation}
\begin{aligned}
	w 
	&\leq A e^2 e^{- A \inf_{M_t} (u - v) + A(u - v) - E_1} \\
	&\leq C e^{A(u - \inf_{M_t} u)} .
\end{aligned}
\end{equation}

\end{proof}

\bigskip

\section{Long time existence and time-independent estimates}
\label{long}
\setcounter{equation}{0}
\medskip

Since $\p_t u$ is bounded, we are able to obtain $C^0$ estimate dependent on time $t$. By the Evans-Krylov theorem and Schauder estimates, we can obtain $C^\infty$ estimates dependent on $t$. Then it is standard to prove the long time existence. To prove convergence, we need time-independent estimates instead. In this section, we obtain time-independent estimates up to second order. Higher order estimates follow from Evans-Krylove theorem and Schauder estimate.

\medskip
\subsection{Hermitian case}

\begin{theorem}
\label{long:theorem-hermitian-c0}
Under the assumption of Theorem~\ref{introduction:theorem-hermitian}, there exists a uniform constant $C$ such that
\begin{equation}
	\sup_M u(x,t) - \inf_M u(x,t) \leq C.
\end{equation}
\end{theorem}

\begin{proof}

We prove the theorem by contradiction. If such a bound does not exist, there is a sequence $t_i \rightarrow \infty$ such that
\begin{equation}
	\sup_M u(x,t_i) - \inf_M u(x,t_i) \rightarrow \infty .
\end{equation}

By \eqref{introduction:theorem-hermitian-condition} and the maximum principle,
\begin{equation}
\label{long:theorem-hermitian-c0-bound-1}
	\p_t u \leq 0 .
\end{equation}
Thus for $t > s \geq 0$,
\begin{equation}
\label{long:theorem-hermitian-c0-bound-2}
	\sup_M u(x,t) \leq \sup_M u(x,s) \leq 0 
\end{equation}
and
\begin{equation}
\label{long:theorem-hermitian-c0-bound-3}
	\inf_M u(x,t) \leq \inf_M u(x,s) \leq 0 
\end{equation}
So we have
\begin{equation}
	\inf_M u(x,t_i) = \inf_{t\in [0,t_i]} \inf_M u(x,t) \rightarrow - \infty.
\end{equation}

By theorem~\ref{C2:theorem-C2-estimate},
\begin{equation}
\begin{aligned}
	w(x,t_i) &\leq C e^{A (u(x,t_i) - \inf_{M \times [0,t_i]} u)} \\
	&= C e^{A (u(x,t_i) - \inf_M u(x,t_i))}.
\end{aligned}
\end{equation}
As shown in \cite{TWv10a}, it follows that
\begin{equation}
	\sup_M u(x,t_i) - \inf_M u(x,t_i) \leq C
\end{equation}
for some positive constant $C$, which contradicts our assumption.

\end{proof}

Following Theorem~\ref{long:theorem-hermitian-c0}, we can show that the $C^2$ estimate is also independent on time $t$.
\begin{corollary}
\label{long:corollary-hermitian-c2}
Under the assumption of Theorem~\ref{introduction:theorem-hermitian}, there exists a uniform constant $C$ such that
\begin{equation}
	w(x,t) \leq C.
\end{equation}
\end{corollary}
\begin{proof}
By \eqref{long:theorem-hermitian-c0-bound-3} and theorem~\ref{C2:theorem-C2-estimate},
\begin{equation}
	w (x,t) \leq C e^{A (u - \inf_{M \times [0,t]} u)} =  C e^{A (u - \inf_{M} u (x,t))}.
\end{equation}
Then the conclusion follows from Theorem~\ref{long:theorem-hermitian-c0}.

\end{proof}

\medskip
\subsection{K\"ahler case}


\begin{theorem}
\label{long:theorem-kahler-c0}
Under the assumption of Theorem~\ref{introduction:theorem-kahler}, there exists a uniform constant $C$ such that
\begin{equation}
	\sup_M u(x,t) - \inf_M u(x,t) \leq C.
\end{equation}
\end{theorem}


First of all, we recall the definition of extended $J$-functionals, which was done in \cite{FLM11, Sun2013p}. Let 
\begin{equation}
	\mathcal{H} := \{u \in C^\infty(M) \,|\, \chi_u > 0\}.
\end{equation}
For any curve $v(s) \in \mathcal{H}$, the functional $J_\a$ is defined by
\begin{equation}
\label{long:definition-J}
	\frac{d J_\a}{d s} = \int_M \frac{\p v}{\p s} \chi^{n - \a}_v \wedge \omega^\a .
\end{equation}
Then we have a formula for $J_\a$ of $u\in\mathcal{H}$,
\begin{equation}
	J_\a (u) = \int^1_0 \int_M \frac{\p v}{\p s} \chi^{n - \a}_v \wedge \omega^\a ds ,
\end{equation}
for any path $v(s) \subset \mathcal{H}$ connecting $0$ and $u$.   The functional is independent on the path. Restricting the integration to the line $v(s) = s u$,
\begin{equation}
\label{long:J-property-1}
\begin{aligned}
	J_\a (u) &= \int^1_0 \int_M u \chi^{n - \a}_{s u} \wedge \omega^\a ds \\
	&= \frac{1}{n - \a + 1} \sum^{n - \a}_{i = 0} \int_M u \chi^i_u \wedge \chi^{n - \a - i} \wedge \omega^\a .
\end{aligned}
\end{equation}
Define a new functional $J$ for equation~\eqref{introduction:parabolic-flow-equation},
\begin{equation}
\label{long:definition-J-new}
	J(u)  := \sum^n_{\a = 1} b_\a J_{\a}(u) = \int^1_0 \int_M \frac{\p v}{\p s} \Big( \sum^n_{\a = 1}  b_\a \chi^{n - \a}_v \wedge \omega^\a \Big) ds ,
\end{equation}
for any path $v(s) \subset \mathcal{H}$ connecting $0$ and $u$. For any $u (x,t)$, we have
\begin{equation}
\label{long:J-new-property-1}
	J (u (T)) = \int^T_0 \int_M \frac{\p u}{\p t} \Big( \sum^n_{\a = 1}  b_\a \chi^{n - \a}_v \wedge \omega^\a \Big)  dt = \int^T_0 \frac{d J}{dt} dt,
\end{equation}
and
and thus along the solution flow $u(x,t)$ to equation~\eqref{introduction:parabolic-flow-equation},
\begin{equation}
\label{long:J-new-property-2}
\begin{aligned}
	\frac{d }{d t} J (u) &= \int_M \p_t u \Big( \sum^n_{\a = 1}  b_\a \chi^{n - \a}_u \wedge \omega^\a \Big) \\
	&= \int_M \Big(\ln \frac{\chi^n_u}{ \sum^n_{\a = 1} b_\a \chi^{n - \a}_u \wedge \omega^\a } - \ln \psi\Big) \Big( \sum^n_{\a = 1}  b_\a \chi^{n - \a}_u \wedge \omega^\a \Big) \\
	&\leq \ln c \int_M \Big( \sum^n_{\a = 1}  b_\a \chi^{n - \a}_u \wedge \omega^\a \Big) - \int_M \ln \psi \Big( \sum^n_{\a = 1}  b_\a \chi^{n - \a}_u \wedge \omega^\a \Big) .
\end{aligned}
\end{equation}
The last inequality follows from Jesen's inequality. The condition $\psi \geq c$ implies that
\begin{equation}
\label{long:J-solution-property-1}
	\frac{d }{d t} J (u) \leq 0.
\end{equation}
This tells us that $J(u)$ is decreasing and non-positive along the solution flow. Let
\begin{equation}
\label{long:J-normalization}
	\hat u = u - \frac{J (u) }{\sum^n_{\a = 1} b_\a \int_M   \chi^{n - \a} \wedge \omega^\a  } .
\end{equation}
Then
\begin{equation}
\label{long:J-normalization-property-1}
	u (x,t) \leq \hat u(x,t).
\end{equation}


\begin{lemma}
\label{long:lemma-kahler}

\begin{equation}
\label{long:lemma-kahler-inequality}
	0 \leq \sup_M \hat u(x,t) \leq - C_1 \inf_M \hat u(x,t) + C_2 ,
\end{equation}
for some constants $C_1$ and $C_2$.

\end{lemma}
\begin{proof}

By \eqref{long:J-new-property-1},
\begin{equation}
\label{long:lemma-kahler-proof-1}
	J (\hat u (T)) = \int^T_0 \int_M \frac{\p \hat u}{\p t} \Big( \sum^n_{\a = 1}  b_\a \chi^{n - \a}_u \wedge \omega^\a \Big) dt = 0.
\end{equation}
Together with \eqref{long:J-property-1}, we have
\begin{equation}
\label{long:lemma-kahler-proof-2}
	\sum^n_{\a = 1} \frac{b_\a}{n - \a + 1} \sum^{n - \a}_{i = 0} \int_M \hat u \chi^i_{\hat u} \wedge \chi^{n - \a - i} \wedge \omega^\a = 0.
\end{equation}
The first inequality in \eqref{long:lemma-kahler-inequality} follows from \eqref{long:lemma-kahler-proof-2}.

Rewriting \eqref{long:lemma-kahler-proof-2},
\begin{equation}	
\label{long:lemma-kahler-proof-3}
	\sum^n_{\a = 1} \frac{b_\a}{n - \a + 1}\int_M \hat u \chi^{n - \a } \wedge \omega^\a = - \sum^{n - 1}_{\a = 1} \frac{b_\a}{n - \a + 1} \sum^{n - \a}_{i = 1} \int_M \hat u \chi^i_{\hat u} \wedge \chi^{n - \a - i} \wedge \omega^\a .
\end{equation}
Let $C$ be a positive constant satisfying
\begin{equation}
\label{long:lemma-kahler-proof-4}
	\omega^n \leq C \ \sum^n_{\a = 1} \frac{b_\a}{n - \a + 1} \chi^{n - \a } \wedge \omega^\a .
\end{equation}
Then
\begin{equation}
\label{long:lemma-kahler-proof-5}
\begin{aligned}
	\int_M \hat u \omega^n &= \int_M \Big(\hat u - \inf_M \hat u\Big) \omega^n + \int_M \inf_M \hat u \omega^n \\
	&\leq C \sum^n_{\a = 1} \frac{b_\a}{n - \a + 1} \int_M \Big(\hat u - \inf_M \hat u\Big) \chi^{n - \a} \wedge \omega^\a + \int_M \inf_M \hat u \omega^n \\
	&= C \sum^n_{\a = 1} \frac{b_\a}{n - \a + 1} \int_M \hat u \chi^{n - \a} \wedge \omega^\a \\
	&\qquad + \inf_M  \hat u \Bigg( \int_M \omega^n -  C \sum^n_{\a = 1} \frac{b_\a}{n - \a + 1} \int_M   \chi^{n - \a} \wedge \omega^\a \Bigg) .
\end{aligned}
\end{equation}
Substituting \eqref{long:lemma-kahler-proof-3} into \eqref{long:lemma-kahler-proof-5},
\begin{equation}
\label{long:lemma-kahler-proof-6}
\begin{aligned}
	\int_M \hat u \omega^n &\leq - C \sum^{n - 1}_{\a = 1} \frac{b_\a}{n - \a + 1} \sum^{n - \a}_{i = 1} \int_M \hat u \chi^i_{\hat u} \wedge \chi^{n - \a - i} \wedge \omega^\a \\
	&\qquad + \inf_M  \hat u \Bigg( \int_M \omega^n -  C \sum^n_{\a = 1} \frac{b_\a}{n - \a + 1} \int_M   \chi^{n - \a} \wedge \omega^\a \Bigg) \\
	&\leq
	\inf_M \hat u \Bigg(\int_M \omega^n - C \sum^n_{\a = 1} b_\a \int_M \chi^{n - \a} \wedge \omega^\a\Bigg).
\end{aligned}
\end{equation}
As shown in \cite{Yau78}, it implies the second inequality in \eqref{long:lemma-kahler-inequality}.

\end{proof}

It remains to prove Theorem~\ref{long:theorem-kahler-c0}.

\begin{proof}[Proof of Theorem~\ref{long:theorem-kahler-c0}]
By Lemma~\ref{long:lemma-kahler}, it is sufficient to show a lower bound for $\inf_M \hat u(x,t)$. If such a lower bound does not exist, then we can choose a sequence $t_i \rightarrow \infty$ such that
\begin{equation}
\label{long:theorem-kahler-proof-1}
	\inf_M \hat u(x,t_i) = \inf_{t \in [0,t_i]} \inf_M \hat u(x,t) ,
\end{equation}
and
\begin{equation}
\label{long:theorem-kahler-proof-2}
	\inf_M \hat u(x,t_i) \rightarrow - \infty .
\end{equation}
By \eqref{long:theorem-kahler-proof-1}, for any $t \leq t_i$,
\begin{equation}
\label{long:theorem-kahler-proof-3}
\begin{aligned}
	&\, \inf_M u (x,t_i) - \inf_M u(x,t) \\
	=&\, \inf_M \hat u (x,t_i) - \inf_M \hat u(x,t) +  \frac{J (u(x,t_i)) - J (u(x,t)) }{\sum^n_{\a = 1} b_\a \int_M   \chi^{n - \a} \wedge \omega^\a  } \\
	\leq&\, \frac{J (u(x,t_i)) - J (u(x,t)) }{\sum^n_{\a = 1} b_\a \int_M   \chi^{n - \a} \wedge \omega^\a  }  .
\end{aligned}
\end{equation}
Since $J (u)$ is decreasing along the solution flow, $\inf_M u (x,t_i) \leq \inf_M u(x,t) $. Together with \eqref{long:J-normalization-property-1}, we have
\begin{equation}
\label{long:theorem-kahler-proof-4}
	\inf_M u(x,t_i) = \inf_{t\in[0,t_i]}\inf_M u(x,t) ,
\end{equation}
and
\begin{equation}
\label{long:theorem-kahler-proof-5}
	\inf_M u(x,t_i) \rightarrow - \infty .
\end{equation}
By theorem~\ref{C2:theorem-C2-estimate},
\begin{equation}
\begin{aligned}
	w(x,t_i) &\leq C e^{A (u(x,t_i) - \inf_{M \times [0,t_i]} u)} \\
	&= C e^{A (u(x,t_i) - \inf_M u(x,t_i))}.
\end{aligned}
\end{equation}
As shown in \cite{TWv10a}, it follows that
\begin{equation}
	\sup_M u(x,t_i) - \inf_M u(x,t_i) \leq C
\end{equation}
for some positive constant $C$, which contradicts our assumption.
\end{proof}

Following Theorem~\ref{long:theorem-kahler-c0}, we can show that the $C^2$ estimate is also independent on time $t$.
\begin{corollary}
\label{long:corollary-kahler-c2}
Under the assumption of Theorem~\ref{introduction:theorem-kahler}, there exists a uniform constant $C$ such that
\begin{equation}
	w(x,t) \leq C .
\end{equation}

\end{corollary}
\begin{proof}

There exists $t_0 \in [0,t]$ such that $\inf_M u (x,t_0) = \inf_{M\times [0,t]} u (x,s)$, and hence
\begin{equation}
\begin{aligned}
	u (x,t) - \inf_{M \times [0,t]} u(x,s) &= u (x,t) - \inf_M u(x,t_0) \\
	&= \hat u(x,t) - \inf_M u(x,t_0) +  \frac{ J (u(x,t)) }{\sum^n_{\a = 1} b_\a \int_M   \chi^{n - \a} \wedge \omega^\a  } \\
	&= \hat u(x,t) - \inf_M \hat u(x,t_0) + \frac{J (u(x,t)) - J (u(x,t_0)) }{\sum^n_{\a = 1} b_\a \int_M   \chi^{n - \a} \wedge \omega^\a  } \\
	&\leq \hat u(x,t) - \inf_M \hat u(x,t_0) .
\end{aligned}
\end{equation}
Together with theorem~\ref{C2:theorem-C2-estimate},
\begin{equation}
	w (x,t) \leq C e^{A (u - \inf_{M \times [0,t]} u)} \leq  C e^{A (\hat u(x,t) - \inf_M \hat u(x,t_0))}.
\end{equation}
Then the conclusion follows from the fact that $\hat u$ is uniformly bounded on $M \times [0,\infty)$, which is in the proof of Theorem~\ref{long:theorem-kahler-c0}.

\end{proof}

\bigskip

\section{The gradient estimate}
\label{gradient}
\setcounter{equation}{0}
\medskip

Although $C^0$ estimate and partial $C^2$ estimate can lead to a gradient estimate, we provide a different argument following \cite{GSun12, SunDissertation}.

\begin{theorem}
\label{gradient:theorem-gradient-estimate}
Let $u \in C^4 (M\times [0, T))$ be an admissible solution to equation~\eqref{introduction:parabolic-flow-equation} and $w = \Delta u + tr\chi$. Suppose that
\begin{equation}
	\mu (x,t) = u (x,t) + H(t), \qquad \frac{dH}{dt} \geq 0.
\end{equation} 
Then there are uniform constants $C$ and $A$ such that
\begin{equation}
	|\nabla u |^2 \leq C e^{A E_2} ,\qquad E_2 = e^{\sup_{M\times [0,t] }(\mu - v) - \inf_{M \times [0,t] } (\mu - v)} -  e^{ \sup_{M\times [0,t]} (\mu - v) - (\mu - v)},
\end{equation}
where $C$, $A$ depend only on geometric data.

\end{theorem}
\begin{proof}

Suppose that the function $ |\nabla u|^2 e^\phi$ attains its maximum at $(p,t_0) \in M_t = M \times (0,t]$. As previous, we choose a normal coordinate chart around $p$. Without loss of generality, we may assume $|\nabla u| \gg 1$ and $|\nabla u| > |\nabla v|$ at $(p,t_0)$. Thus we have the following relations at $(p,t_0)$,
\begin{equation}
\label{gradient:test-derivative-1}
	\frac{\p_i (|\nabla u|^2)}{|\nabla u|^2} + \p_i \phi = 0 ,
\end{equation}
\begin{equation}
\label{gradient:test-derivative-2}
	\frac{\bpartial_i (|\nabla u|^2)}{|\nabla u|^2} + \bpartial_i \phi = 0 ,
\end{equation}
\begin{equation}
\label{gradient:test-derivative-3}
	\frac{\p_t (|\nabla u|^2)}{|\nabla u|^2} + \p_t \phi \geq 0 ,
\end{equation}
and
\begin{equation}
\label{gradient:test-derivative-4}
	\frac{\bpartial_i\p_i |\nabla u|^2}{|\nabla u|^2} - \frac{\p_i (|\nabla u|^2) \bpartial_i (|\nabla u|^2)}{|\nabla u|^4} + \bpartial_i\p_i \phi \leq 0.
\end{equation}
Direct computation shows that,
\begin{equation}
\label{gradient:test-derivative-5}
	\p_i (|\nabla u|^2) = \sum_k (u_k u_{i\bar k} + u_{ki} u_{\bar k}),
\end{equation}
\begin{equation}
\label{gradient:test-derivative-6}
	\p_t (|\nabla u|^2) = \sum_k (\p_t u_k u_{\bar k} + u_k \p_t u_{\bar k}),
\end{equation}
and
\begin{equation}
\label{gradient:test-derivative-7}
\begin{aligned}
	\bpartial_i \p_i (|\nabla u|^2) =&\, \sum_k (u_{ki} u_{\bar k\bar i} + u_{i\bar ik} u_{\bar k} + u_{i\bar i\bar k} u_k) + \sum_{k,l} R_{i\bar i k\bar l} u_l u_{\bar k} \\
	&\, \sum_k \Big|u_{\bar ki} - \sum_l T^k_{il} u_{\bar l}\Big|^2 - \sum_k \Big|\sum_l T^k_{il} u_{\bar l}\Big|^2 .
\end{aligned}
\end{equation}
By \eqref{gradient:test-derivative-1} and \eqref{gradient:test-derivative-6},
\begin{equation}
\label{gradient:test-derivative-8}
\begin{aligned}
	|\p_i(|\nabla u|^2)|^2 
	&= \Big|\sum_k u_{ki} u_{\bar k}\Big|^2 - 2 |\nabla u|^2 \mathfrak{Re} \Big\{\sum_k u_k u_{\bar ki} \bpartial_i\phi\Big\} - \Big|\sum_k u_k u_{\bar ki}\Big|^2 . 
\end{aligned}
\end{equation}
By \eqref{gradient:test-derivative-7} and Schwarz inequality,
\begin{equation}
\label{gradient:theorem-proof-1}
\begin{aligned}
	&\, \sum_i S_{\a - 1;i} (X^{i\bar i})^ 2 \bpartial_i \p_i (|\nabla u|^2) \\
	\geq&\, \frac{1}{|\nabla u|^2} \sum_i S_{\a - 1;i} (X^{i\bar i})^2 \Big|\sum_k u_{ki} u_{\bar k}\Big|^2 - \sum_k \p_k S_\a u_{\bar k} - \sum_k \bpartial_k S_\a u_k \\
	&\, - \sum_{i,k} S_{\a - 1;i} (X^{i\bar i})^2 \chi_{i\bar ik} u_{\bar k} - \sum_{i,k} S_{\a - 1;i} (X^{i\bar i})^2 \chi_{i\bar i\bar k} u_k \\
	&\, + \sum_{i,k,l} S_{\a - 1;i} (X^{i\bar i})^ 2 R_{i\bar i k\bar l} u_l u_{\bar k}  - \sum_{i,k} S_{\a - 1;i} (X^{i\bar i})^ 2  \Big|\sum_l T^k_{il} u_{\bar l}\Big|^2 \\
	\geq&\, \frac{1}{|\nabla u|^2} \sum_i S_{\a - 1;i} (X^{i\bar i})^2 \Big|\sum_k u_{ki} u_{\bar k}\Big|^2 - \sum_k \p_k S_\a u_{\bar k} - \sum_k \bpartial_k S_\a u_k \\
	&\, - C_2 |\nabla u|^2 \sum_{i} S_{\a - 1;i} (X^{i\bar i})^2  - C_2 |\nabla u| \sum_{i} S_{\a - 1;i} (X^{i\bar i})^2 ,
\end{aligned}
\end{equation}
for some uniform constant $C_2 > 0$; by \eqref{gradient:test-derivative-8}, 
\begin{equation}
\label{gradient:theorem-proof-2}
\begin{aligned}
	&\, \sum_i S_{\a - 1;i} (X^{i\bar i})^2 |\p_i (|\nabla u|^2)|^2 \\
	\leq&\,   \sum_i S_{\a - 1;i} (X^{i\bar i})^2 \Big|\sum_k u_{ki} u_{\bar k}\Big|^2 - 2 |\nabla u|^2 \sum_i S_{\a - 1;i} (X^{i\bar i})^2 \mathfrak{Re} \Big\{\sum_k u_k u_{\bar ki} \bpartial_i\phi\Big\} .
\end{aligned}
\end{equation}
Substituting \eqref{gradient:theorem-proof-1} and \eqref{gradient:theorem-proof-2} into \eqref{gradient:test-derivative-4},
\begin{equation}
\label{gradient:theorem-proof-3}
\begin{aligned}
	&  |\nabla u|^2 \sum_i S_{\a - 1;i} (X^{i\bar i})^2\bpartial_i\p_i \phi +  2 \sum_i S_{\a - 1;i} (X^{i\bar i})^2 \mathfrak{Re} \Big\{\sum_k u_k u_{\bar ki} \bpartial_i\phi\Big\}  \\
	 \leq&\,  \sum_k \p_k S_\a u_{\bar k} +  \sum_k \bpartial_k S_\a u_k + C_2 ( |\nabla u|^2 + |\nabla u|) \sum_{i} S_{\a - 1;i} (X^{i\bar i})^2 .
\end{aligned}
\end{equation}

Let $\phi = A e^\eta$ , where $\eta = v - \mu + \sup_{M\times [0,t]} (\mu - v)\geq 0$ and $A$ is to be determined. It is easy to see that
\begin{equation}
\label{gradient:phi-derivative-1}
	\p_t \phi = \phi \p_t \eta =  - \phi \p_t \mu,
\end{equation}
\begin{equation}
\label{gradient:phi-derivative-2}
	\p_i \phi = \phi \eta_i ,
\end{equation}
\begin{equation}
\label{gradient:phi-derivative-2-conjugate}
	\bpartial_i \phi = \phi \eta_{\bar i} ,
\end{equation}
and
\begin{equation}
\label{gradient:phi-derivative-3}
	\bpartial_i\p_i \phi = \phi (\eta_i \eta_{\bar i} + \eta_{i\bar i}).
\end{equation}
Then from \eqref{gradient:theorem-proof-3},
\begin{equation}
\label{gradient:theorem-proof-4}
\begin{aligned}
	&\,  |\nabla u|^2 \phi \sum_i S_{\a - 1;i} (X^{i\bar i})^2 (\eta_i \eta_{\bar i} + \eta_{i\bar i}) +  2 \phi \sum_i S_{\a - 1;i} (X^{i\bar i})^2 \mathfrak{Re} \Big\{\sum_k u_k u_{\bar ki} \eta_{\bar i}\Big\}  \\
	\leq&\,  \sum_k \p_k S_\a u_{\bar k} +  \sum_k \bpartial_k S_\a u_k + C_2 (|\nabla u|^2  + |\nabla u)|\sum_{i} S_{\a - 1;i} (X^{i\bar i})^2 .
\end{aligned}
\end{equation}
Controlling the second term,
\begin{equation}
\label{gradient:theorem-proof-5}
\begin{aligned}
	&\, 2 \phi \sum_i S_{\a - 1;i} (X^{i\bar i})^2 \mathfrak{Re} \Big\{\sum_k u_k u_{\bar ki} \eta_{\bar i}\Big\} \\
	=&\, 2 \phi \sum_i S_{\a - 1;i} X^{i\bar i} \mathfrak{Re} \Big\{u_i \eta_{\bar i}\Big\} - 2 \phi \sum_i S_{\a - 1;i} (X^{i\bar i})^2 \mathfrak{Re} \Big\{\sum_k u_k \chi_{\bar ki} \eta_{\bar i}\Big\} \\
	\geq&\, 2 \phi \sum_i S_{\a - 1;i} X^{i\bar i} \mathfrak{Re} \Big\{u_i \eta_{\bar i}\Big\} - \frac{\phi}{2} |\nabla u|^2 \sum_i S_{\a - 1;i} (X^{i\bar i})^2 |\eta_i|^2 \\
	&\hspace{12em} - C_3 \phi \sum_i S_{\a - 1;i} (X^{i\bar i})^2 .
\end{aligned}
\end{equation}
Substituting \eqref{gradient:theorem-proof-5} into \eqref{gradient:theorem-proof-4},
\begin{equation}
\label{gradient:theorem-proof-6}
\begin{aligned}
	0 \leq&\,  \sum_k \p_k S_\a u_{\bar k} +  \sum_k \bpartial_k S_\a u_k + \phi |\nabla u|^2 \sum_i S_{\a - 1;i} (X^{i\bar i})^2 (u_{i\bar i} - v_{i\bar i}) \\
	&\, -   2 \phi \sum_i S_{\a - 1;i} X^{i\bar i} \mathfrak{Re} \Big\{u_i \eta_{\bar i}\Big\} - \frac{\phi}{2} |\nabla u|^2 \sum_i S_{\a - 1;i} (X^{i\bar i})^2 |\eta_i|^2 \\
	&\, + C_2 (|\nabla u|^2 + |\nabla u|)\sum_{i} S_{\a - 1;i} (X^{i\bar i})^2 + C_3 \phi \sum_i S_{\a - 1;i} (X^{i\bar i})^2 .
\end{aligned}
\end{equation}
Noting that 
\begin{equation}
\label{gradient:test-derivative-9}
	\p_t (|\nabla u|^2) \leq C_4 |\nabla u| -  \sum_k \p_k \Big(\ln\sum^n_{\a = 1} c_\a S_\a \Big)  u_{\bar k} -  \sum_k \bpartial_k \Big(\ln\sum^n_{\a = 1} c_\a S_\a \Big) u_k,
\end{equation}
then
by \eqref{gradient:test-derivative-3} and \eqref{gradient:phi-derivative-1},
\begin{equation}
\label{gradient:test-derivative-10}
\begin{aligned}
	 \sum_k \p_k \Big(\ln\sum^n_{\a = 1} c_\a S_\a \Big)  u_{\bar k} +  \sum_k \bpartial_k \Big(\ln\sum^n_{\a = 1} c_\a S_\a \Big) u_k 
	 &\leq  C_4 |\nabla u| + |\nabla u|^2 \p_t \phi \\
	 &\leq C_4 |\nabla u| - \phi |\nabla u|^2 \p_t u  .
\end{aligned}
\end{equation}
Putting together \eqref{gradient:theorem-proof-6} and \eqref{gradient:test-derivative-10},
\begin{equation}
\label{gradient:theorem-proof-7}
\begin{aligned}
	0 \leq&\, \ln \Big(\psi\sum^n_{\a = 1} c_\a S_\a \Big) +  \sum^n_{\a = 1}\frac{c_\a}{\sum^n_{\a = 1} c_\a S_\a} \sum_i S_{\a - 1;i} (X^{i\bar i})^2 (u_{i\bar i} - v_{i\bar i}) \\
	&\qquad - \frac{2}{|\nabla u|^2}  \sum^n_{\a = 1}\frac{c_\a}{\sum^n_{\a = 1} c_\a S_\a}  \sum_i S_{\a - 1;i} X^{i\bar i}  \mathfrak{Re} \Big\{u_i \eta_{\bar i}\Big\} \\
	&\qquad - \frac{1}{2}  \sum^n_{\a = 1}\frac{c_\a}{\sum^n_{\a = 1} c_\a S_\a}  \sum_i S_{\a - 1;i} (X^{i\bar i})^2 |\eta_i|^2  \\
	&\qquad + \frac{C_2 ( |\nabla u| + 1)}{\phi |\nabla u|} \sum^n_{\a = 1}\frac{c_\a}{\sum^n_{\a = 1} c_\a S_\a} \sum_{i} S_{\a - 1;i} (X^{i\bar i})^2   \\
	&\qquad + \frac{C_3}{|\nabla u|^2} \sum^n_{\a = 1}\frac{c_\a}{\sum^n_{\a = 1} c_\a S_\a} \sum_i S_{\a - 1;i} (X^{i\bar i})^2 + \frac{C_4}{\phi |\nabla u| }.
\end{aligned}
\end{equation}
By the fact that $\phi \geq A$,
\begin{equation}
\label{gradient:theorem-proof-8}
\begin{aligned}
	0 \leq&\, \ln \Big(\psi\sum^n_{\a = 1} c_\a S_\a \Big) +  \sum^n_{\a = 1}\frac{c_\a}{\sum^n_{\a = 1} c_\a S_\a} \sum_i S_{\a - 1;i} (X^{i\bar i})^2 (u_{i\bar i} - v_{i\bar i}) \\
	&\, - \frac{2}{|\nabla u|^2}  \sum^n_{\a = 1}\frac{c_\a}{\sum^n_{\a = 1} c_\a S_\a}  \sum_i S_{\a - 1;i} X^{i\bar i}  \mathfrak{Re} \Big\{u_i \eta_{\bar i}\Big\} \\
	&\,  - \frac{1}{2}  \sum^n_{\a = 1}\frac{c_\a}{\sum^n_{\a = 1} c_\a S_\a}  \sum_i S_{\a - 1;i} (X^{i\bar i})^2 |\eta_i|^2  + \frac{C_4}{A |\nabla u| } \\
	&\, + C_5\Bigg(\frac{1}{A} + \frac{1}{A|\nabla u|} + \frac{1}{|\nabla u|^2}\Bigg) \sum^n_{\a = 1}\frac{c_\a}{\sum^n_{\a = 1} c_\a S_\a} \sum_{i} S_{\a - 1;i} (X^{i\bar i})^2 .
\end{aligned}
\end{equation}
Since we assume $|\nabla u| > |\nabla v|$,
\begin{equation}
\label{gradient:theorem-proof-9}
\begin{aligned}
	- \frac{2}{|\nabla u|^2}  \sum^n_{\a = 1}\frac{c_\a}{\sum^n_{\a = 1} c_\a S_\a}  \sum_i S_{\a - 1;i} X^{i\bar i} \mathfrak{Re} \Big\{u_i \eta_{\bar i}\Big\} \\
	\leq  4 \sum^n_{\a = 1} \frac{c_\a}{\sum^n_{\a = 1} c_\a S_\a}  \sum_i S_{\a - 1;i} X^{i\bar i} \leq 4n .
\end{aligned}
\end{equation}

Case 1. $w > N$, where $N$ is the crucial constant in Lemma~\ref{lemma-inequality}. We have
\begin{equation}
\label{gradient:theorem-proof-10}
\begin{aligned}
	&\,   \theta + \theta \sum^n_{\a = 1} \frac{c_\a}{\sum^n_{\a = 1} c_\a S_\a} \sum_i S_{\a - 1;i} (X^{i\bar i})^2 \\
	&\, + \frac{1}{2}  \sum^n_{\a = 1}\frac{c_\a}{\sum^n_{\a = 1} c_\a S_\a}  \sum_i S_{\a - 1;i} (X^{i\bar i})^2 |\eta_i|^2\\
	\leq&\,  - \frac{2}{|\nabla u|^2}  \sum^n_{\a = 1}\frac{c_\a}{\sum^n_{\a = 1} c_\a S_\a}  \sum_i S_{\a - 1;i} X^{i\bar i}  \mathfrak{Re} \Big\{u_i \eta_{\bar i}\Big\}  + \frac{C_4}{A |\nabla u| }\\
	&\, + C_5\Bigg(\frac{1}{A} + \frac{1}{A|\nabla u|} + \frac{1}{|\nabla u|^2}\Bigg) \sum^n_{\a = 1}\frac{c_\a}{\sum^n_{\a = 1} c_\a S_\a} \sum_{i} S_{\a - 1;i} (X^{i\bar i})^2 .
\end{aligned}
\end{equation}

If $\sum^n_{\a = 1}\frac{c_\a}{\sum^n_{\a = 1} c_\a S_\a}  S_{\a - 1;i} (X^{i\bar i})^2 \geq K$ for some $i$ and $K > 0$,
\begin{equation}
\label{gradient:theorem-proof-11}
\begin{aligned}
	&\,   \theta + \frac{\theta}{2} \sum^n_{\a = 1} \frac{c_\a}{\sum^n_{\a = 1} c_\a S_\a} \sum_i S_{\a - 1;i} (X^{i\bar i})^2 +  \frac{\theta K}{2}\\
	\leq&\, 4n + \frac{C_4}{A |\nabla u| } + C_5\Bigg(\frac{1}{A} + \frac{1}{A|\nabla u|} + \frac{1}{|\nabla u|^2}\Bigg) \sum^n_{\a = 1}\frac{c_\a}{\sum^n_{\a = 1} c_\a S_\a} \sum_{i} S_{\a - 1;i} (X^{i\bar i})^2 .
\end{aligned}
\end{equation}
Choosing $\theta$ satisfying $\theta K > 8n$,
\begin{equation}
\label{gradient:theorem-proof-12}
\begin{aligned}
	&\,   \theta + \frac{\theta}{2} \sum^n_{\a = 1} \frac{c_\a}{\sum^n_{\a = 1} c_\a S_\a} \sum_i S_{\a - 1;i} (X^{i\bar i})^2 \\
	\leq&\, \frac{C_4}{A |\nabla u| } + C_5\Bigg(\frac{1}{A} + \frac{1}{A|\nabla u|} + \frac{1}{|\nabla u|^2}\Bigg) \sum^n_{\a = 1}\frac{c_\a}{\sum^n_{\a = 1} c_\a S_\a} \sum_{i} S_{\a - 1;i} (X^{i\bar i})^2 .
\end{aligned}
\end{equation}
When $A$ is large enough, we derive a bound $|\nabla u| \leq C$.

If $\sum^n_{\a = 1}\frac{c_\a}{\sum^n_{\a = 1} c_\a S_\a}  S_{\a - 1;i} (X^{i\bar i})^2 \leq K$ for all $i$, 
then
\begin{equation}
\label{gradient:theorem-proof-13}
\begin{aligned}
	\sum^n_{\a = 1} \frac{c_\a \a S_\a S_1}{\sum^n_{\a = 1} c_\a S_\a} \leq n^2 K ,
\end{aligned}
\end{equation}
and hence
\begin{equation}
\label{gradient:theorem-proof-14}
	S_1 \leq n^2 K .
\end{equation}
By Schwarz inequality and \eqref{gradient:theorem-proof-14},
\begin{equation}
\label{gradient:theorem-proof-15}
\begin{aligned}
	&\, 
	- \frac{2}{|\nabla u|^2}  \sum_i S_{\a - 1;i} X^{i\bar i}  \mathfrak{Re} \Big\{u_i \eta_{\bar i}\Big\} \\
	\leq&\,
	\frac{4(n - \a + 1)}{|\nabla u|^2} S_{\a - 1} + \frac{1}{4}  \sum_i S_{\a - 1;i} (X^{i\bar i})^2 |\eta_i|^2 .
\end{aligned}
\end{equation}
Substituting \eqref{gradient:theorem-proof-15} into \eqref{gradient:theorem-proof-10},
\begin{equation}
\label{gradient:theorem-proof-16}
\begin{aligned}
	&\,   \theta + \theta \sum^n_{\a = 1} \frac{c_\a}{\sum^n_{\a = 1} c_\a S_\a} \sum_i S_{\a - 1;i} (X^{i\bar i})^2 \\
	\leq&\,   C_5\Bigg(\frac{1}{A} + \frac{1}{A|\nabla u|} + \frac{1}{|\nabla u|^2}\Bigg) \sum^n_{\a = 1}\frac{c_\a}{\sum^n_{\a = 1} c_\a S_\a} \sum_{i} S_{\a - 1;i} (X^{i\bar i})^2 \\
	&\,\hspace{6em}  + \frac{C_6}{|\nabla u|^2} + \frac{C_4}{A |\nabla u| }  .
\end{aligned}
\end{equation}
This leads to a bound for $|\nabla u|$ if $A$ is chosen large enough.

Case 2. $w \leq N$. We have
\begin{equation}
\label{gradient:theorem-proof-17}
	\sum_i S_{\a - 1;i} (X^{i\bar i})^2 |\eta_i|^2 \geq |\nabla \eta|^2 \min_i S_{\a - 1;i} (X^{i\bar i})^2 \geq \frac{C^{\a - 1}_{n - 1}}{w^{\a + 1}} |\nabla \eta|^2 \geq \frac{C^{\a - 1}_{n - 1}}{N^{\a + 1}} |\nabla \eta|^2 .
\end{equation}
Substituting \eqref{gradient:theorem-proof-9} and \eqref{gradient:theorem-proof-17} into \eqref{gradient:theorem-proof-8},
\begin{equation}
\label{gradient:theorem-proof-18}
\begin{aligned}
	\Bigg(\epsilon - \frac{C_5}{A} - \frac{C_5}{A|\nabla u|} - \frac{C_5}{|\nabla u|^2} \Bigg) \sum^n_{\a = 1} \frac{c_\a}{\sum^n_{\a = 1} c_\a S_\a} \sum_i S_{\a - 1;i} (X^{i\bar i})^2 \\
	\leq \ln \Big(\psi\sum^n_{\a = 1} c_\a S_\a \Big)   + 5n  + \frac{C_4}{A |\nabla u| } - \epsilon_1 |\nabla \eta|^2 .
\end{aligned}
\end{equation}
Choosing $A$ sufficiently large and noting the assumption that $|\nabla u| \gg 1$,
\begin{equation}
\label{gradient:theorem-proof-19}
	\epsilon_1 |\nabla \eta|^2 + \Bigg(\frac{\epsilon}{2} - \frac{C_5}{|\nabla u|^2} \Bigg) \sum^n_{\a = 1} \frac{c_\a}{\sum^n_{\a = 1} c_\a S_\a} \sum_i S_{\a - 1;i} (X^{i\bar i})^2 \leq C_7  .
\end{equation}
When $|\nabla u|$ is larger than $\sqrt{2 C_5 / \epsilon}$, we derive a bound for $|\nabla \eta|$. But
\begin{equation}
	|\nabla u| = |\nabla \mu| \leq |\nabla \eta | + |\nabla v| .
\end{equation}

\end{proof}

Applying appropriate $H(t)$, we can obtain time-independent gradient estimate. 
For Hermitian manifolds, we have the following corollary.
\begin{corollary}
\label{gradient:corollary-hermitian}
Under the assumption of Theorem~\ref{introduction:theorem-hermitian}, there exists a uniform constant $C$ such that
\begin{equation}
	|\nabla u|^2 \leq C.
\end{equation}

\end{corollary}

\begin{proof}
Setting $\mu = \tilde u$,
\begin{equation}
	H(t) = - \frac{\int_M u \omega^n}{\int_M \omega^n}
\end{equation}
and thus
\begin{equation}
	\frac{dH}{dt} = - \frac{\int_M \p_t u \omega^n}{\int_M \omega^n} \geq 0.
\end{equation}
By theorem~\ref{gradient:theorem-gradient-estimate}, there are constants $C$ and $A$ such that
\begin{equation}
	|\nabla u |^2 \leq C e^{A E_2} ,\qquad E_2 = e^{\sup_{M\times [0,t] }(\tilde u - v) - \inf_{M \times [0,t] } (\tilde u  - v)} -  e^{ \sup_{M\times [0,t]} (\tilde u - v) - (\tilde u - v)}.
\end{equation}
Theorem~\ref{long:theorem-hermitian-c0} implies that $\tilde u$ is uniformly bounded, and then we obtain the bound for $|\nabla u|^2$.

\end{proof}

For K\"ahler manifolds, we have the following corollary.
\begin{corollary}
\label{gradient:corollary-kahler}
Under the assumption of Theorem~\ref{introduction:theorem-kahler}, there exists a uniform constant $C$ such that
\begin{equation}
	|\nabla u|^2 \leq C.
\end{equation}

\end{corollary}

\begin{proof}
Setting $\mu = \hat u$,
\begin{equation}
	H(t) = - \frac{J(u)}{\sum^n_{\a = 1} b_\a \int_M \chi^{n - \a} \wedge \omega^\a}
\end{equation}
and thus
\begin{equation}
	\frac{dH}{dt} = - \frac{\int_M \p_t u \Big(\sum^n_{\a = 1} b_\a \chi^{n - \a}_u \wedge \omega^\a\Big)}{\sum^n_{\a = 1} b_\a \int_M \chi^{n - \a} \wedge \omega^\a} \geq 0.
\end{equation}
By theorem~\ref{gradient:theorem-gradient-estimate}, there are constants $C$ and $A$ such that
\begin{equation}
	|\nabla u |^2 \leq C e^{A E_2} ,\qquad E_2 = e^{\sup_{M\times [0,t] }(\hat u - v) - \inf_{M \times [0,t] } (\hat u  - v)} -  e^{ \sup_{M\times [0,t]} (\hat u - v) - (\hat u - v)}.
\end{equation}
By \eqref{long:lemma-kahler-proof-2} in the proof of Lemma~\ref{long:lemma-kahler}, we have
\begin{equation}
	\inf_M \hat u(x,t) \leq 0 \leq \sup_M \hat u(x,t).
\end{equation} 
Then Theorem~\ref{long:theorem-kahler-c0} implies that $\tilde u$ is uniformly bounded, and consequently we obtain the bound for $|\nabla u|^2$.

\end{proof}

\bigskip

\section{Convergence of the parabolic flow}
\label{convergence}
\setcounter{equation}{0}
\medskip

In this section, we prove the convergence of the normalized solution. The arguments of Gill~\cite{Gill11} can be applied verbatim here, so we just give a sketch.

Let $\varphi$ be a positive function on $M \times (0,\infty)$ such that
\begin{equation}
	\frac{\p\varphi}{\p t} = \sum_{i,j} G^{i\bar j} \bpartial_j\p_i \varphi,
\end{equation}
where
\begin{equation}
	G^{i\bar j} = \frac{\p}{\p u_{i\bar j}} \Bigg(\ln \frac{S_n ( \chi_u )}{\sum^n_{\a = 1} c_\a S_{n - \a}(\chi_u)}\Bigg) .
\end{equation}
At a point $p$ and under a normal coordinate chart, $\{G^{i\bar j}\}$ is diagonal and
\begin{equation}
	G^{i\bar i} = \sum^n_{\a = 1}\frac{c_\a}{\sum^n_{\a = 1} c_\a S_\a} \sum_i S_{\a - 1;i} (X^{i\bar i})^2 .
\end{equation}
We have the following Harnack inequality.
\begin{lemma}
\label{convergence:lemma-harnack}
For all $0 < t_1 < t_2$,
\begin{equation}
	\sup_{x\in M} \varphi (x,t_1) \leq \inf_{x\in M} \varphi(x,t_2) \Big(\frac{t_2}{t_1}\Big)^{C_2} e^{\frac{C_3}{t_2 - t_1} + C_1(t_2 - t_1)}
\end{equation}
\end{lemma}

Now we apply the argument of Cao~\cite{Cao85}. Define $\varphi = \p_t u$. Let $m$ be a positive integer and define
\begin{equation}
\begin{aligned}
	\varphi'_m (x,t) = \sup_{y\in M} \varphi (y, m - 1) - \varphi(x, m - 1 + t),\\
	\varphi''_m(x,t) = \varphi (x, m - 1 + t) - \inf_{y\in M} \varphi(y, m - 1).
\end{aligned}
\end{equation}
By \eqref{preliminary:equation-time-derivative}, we have
\begin{equation}
\begin{aligned}
	\frac{\p \varphi'_m}{\p t} = \sum_{i,j} G^{i\bar j}(m - 1 + t) \bpartial_j\p_i \varphi'_m ,\\
	\frac{\p \varphi''_m}{\p t}  = \sum_{i,j} G^{i\bar j}(m - 1 + t) \bpartial_j\p_i \varphi''_t .
\end{aligned}
\end{equation}
If $\varphi(x, m - 1)$ is not a constant function, $\varphi'$ and $\varphi''$ must be positive functions on $M \times (0,\infty)$ by the maximum principle. Applying Lemma~\ref{convergence:lemma-harnack} with $t_1 = \frac{1}{2}$ and $t_2 = 1$, we have
\begin{equation}
	O (m - 1) + O \Big(m - \frac{1}{2}\Big) \leq C (O (m - 1) - O (m)),
\end{equation}
where $O (t) = \sup_M \varphi (x,t) - \inf_M \varphi (x,t)$. Therefore $O (t) \leq C e^{- c_0 t}$ for $c_0 > 0$. On the other hand, this inequality is also true if $\varphi(x, m - 1)$ is constant.

For $(x,t) \in M \times [0,\infty)$, there is a point $y \in M$ such that $\p_t \tilde u = 0$ since $\int_M \p_t \tilde u \omega^n = 0$. Therefore,
\begin{equation}
\label{convergence:solution-normalization-decay}
	|\p_t \tilde u (x,t)| \leq C e^{- c_0 t} ,
\end{equation}
and consequently
\begin{equation}
	\p_t Q \leq 0 ,
\end{equation}
where $Q = \tilde u + \frac{C}{c_0} e^{- c_0 t}$.
Since $Q$ is bounded and decreasing,
\begin{equation}
	\lim_{t\rightarrow \infty} \tilde u = \lim_{t\rightarrow \infty} Q = \tilde u_\infty .
\end{equation}
By contradiction argument, we can prove that the convergence of $\tilde u$ to $\tilde u_\infty$ is actually $C^\infty$.

Note that $\tilde u$ satisfies
\begin{equation}
\label{convergence:solution-normalization-flow}
	\frac{\p \tilde u}{\p t} + \frac{\int_M \p_t u \omega^n}{\int_M \omega^n} = \ln \frac{\chi^n_{\tilde u}}{ \sum^n_{\a = 1} b_\a \chi^{n - \a}_{\tilde u} \wedge \omega^\a } - \ln \psi.
\end{equation}
Letting $t\rightarrow\infty$, the right side of \eqref{convergence:solution-normalization-flow} tells us that it converges. Then \eqref{convergence:solution-normalization-decay} implies that it converges to a constant $b$. This completes the proof of the convergence.

\bigskip

\section{Generalized \texorpdfstring{$J$}{J}-flow in K\"ahler geometry}
\label{J}
\setcounter{equation}{0}
\medskip

In this section, we just show some key steps of the proof, that is, the second order estimate and the uniform estimate. The remaining parts follow from the previous sections.

Rewriting \eqref{introduction:parabolic-flow-equation-J},
\begin{equation}
\label{J:parabolic-flow-equation-J-1}
	\frac{\p u }{\p t} = \frac{1}{c} -  \sum^n_{\a = 1} c_\a S_\a  .
\end{equation}
Since $\omega$ is K\"ahler,
\begin{equation}
\label{J:formula-X-1}
\begin{aligned}
     	X_{i\bar ij\bar j} - X_{j\bar ji\bar i} &= R_{j\bar ji\bar i}X_{i\bar i}  -  R_{i\bar ij\bar j}X_{j\bar j} - G_{i\bar ij\bar j},
\end{aligned}
\end{equation}
where
\begin{equation}
\label{J:formula-G-coefficient}
\begin{aligned}
    	G_{i\bar ij\bar j} &= \chi_{j\bar ji\bar i} - \chi_{i\bar ij\bar j} + \sum_p R_{j\bar ji\bar p}\chi_{p\bar i} -\sum_p R_{i\bar ij\bar p}\chi_{p\bar j} .
\end{aligned}
\end{equation}
Differentiating the equation at $p$
\begin{equation}
\label{J:equation-time-derivative}
	\p_t (\p_t u) = \sum_i S_{\a - 1;i} (X^{i\bar i})^2 (\p_t u)_{i\bar i} ,
\end{equation}
\begin{equation}
\label{J:equation-1st-derivative}
	\p_t u_l = \sum^n_{\a = 1} c_\a \sum_i S_{\a - 1;i} (X^{i\bar i})^2 X_{i\bar il}  ,
\end{equation}
and
\begin{equation}
\label{J:equation-2nd-derivative}
\begin{aligned}
	\p_t u_{l\bar l}	&\leq 
	\sum^n_{\a = 1} c_\a \Bigg[ \sum_i S_{\a - 1;i} (X^{i\bar i})^2 X_{i\bar il\bar l} - \sum_{i,j} S_{\a - 1;i} (X^{i\bar i})^2 X^{j\bar j} X_{j\bar i\bar l} X_{i\bar jl} \Bigg].
\end{aligned}
\end{equation}

\medskip
\subsection{Second order estimate}

\begin{theorem}
\label{J:theorem-C2-estimate}
Let $u \in C^4 (M\times [0, T))$ be an admissible solution to equation~\eqref{introduction:parabolic-flow-equation-J} and $w = \Delta u + tr\chi$. Then there are uniform constants $C$ and $A$ such that
\begin{equation}
	w \leq C e^{A (u - \inf_{M\times [ 0, t]} u)},
\end{equation}
where $C$, $A$ depend only on geometric data.
\end{theorem}

\begin{proof}
We consider the function $w e^\phi$ where $\phi$, where
\begin{equation}
\label{J:C2-definition-phi}
	\phi := - A (u - v)  .
\end{equation}
and A is to determined. Suppose that $\ln w e^\phi$ achieves its maximum at some point $(p,t_0) \in M_t = M \times (0,t]$. Without loss of generality, we may assume that $w \gg 1$. Choose a local chart around $p$ such that $g_{i\bar j} = \delta_{ij}$ and $X_{i\bar j}$ is diagonal at $p$ when $t = t_0$. Therefore, we have at the point $(p,t_0)$,
\begin{equation}
\label{J:C2-test-derivative-1}
	\frac{\p_l w}{w} + \p_l \phi = 0 ,
\end{equation}
\begin{equation}
\label{J:C2-test-derivative-2}
	\frac{\bpartial_l w}{w} + \bpartial_l \phi = 0 ,
\end{equation}
\begin{equation}
\label{J:C2-test-derivative-3}
	\frac{\p_t w}{w} + \p_t \phi \geq 0 ,
\end{equation}
and
\begin{equation}
\label{J:C2-test-derivative-4}
	\frac{\bpartial_l\p_l w}{w} - \frac{|\p_l w |^2 }{w^2} + \bpartial_l\p_l \phi \leq 0 .
\end{equation}

Note that
\begin{equation}
\label{J:C2-square-1}
\begin{aligned}
	 \sum_i S_{\a - 1;i} (X^{i\bar i})^2 \frac{|\p_i w|^2}{w} \leq \sum_{i,j,l} S_{\a - 1;i} (X^{i\bar i})^2 X^{j\bar j} X_{i\bar jl} X_{j\bar i\bar l} .
\end{aligned}
\end{equation}
By \eqref{J:formula-X-1}, \eqref{J:C2-square-1} and \eqref{C2:inequality-1},
\begin{equation}
\label{J:C2-inequality-1}
\begin{aligned}
	&\, \sum_{i,j,l} S_{\a - 1;i} (X^{i\bar i})^2 X^{j\bar j} X_{j\bar i\bar l} X_{i\bar jl} - \sum_{i,l} S_{\a - 1;i} (X^{i\bar i})^2 X_{i\bar il\bar l} \\
	\geq&\, w \sum_i S_{\a - 1;i} (X^{i\bar i})^2 \bpartial_i \p_i \phi - C_1 w \sum_{i,l} S_{\a - 1;i} (X^{i\bar i})^2 - C_2 \frac{\a}{n} S_\a.
\end{aligned}
\end{equation}
Combining \eqref{J:equation-2nd-derivative} and \eqref{J:C2-inequality-1},
\begin{equation}
\label{J:C2-inequality-3}
\begin{aligned}
	\p_t w 
	\leq&\, - w \sum^n_{\a = 1}c_\a  \sum_i S_{\a - 1;i} (X^{i\bar i})^2 \bpartial_i \p_i \phi \\
	&\,  + C_1 w \sum^n_{\a = 1}c_\a \sum_{i} S_{\a - 1;i} (X^{i\bar i})^2  + C_2 \sup_M \sum^n_{\a = 1} c_\a S_\a (\chi^{-1}).
\end{aligned}
\end{equation}
By \eqref{J:C2-test-derivative-2} and \eqref{J:C2-test-derivative-3},
\begin{equation}
\label{J:C2-inequality-5}
\begin{aligned}
	w \p_t\phi \geq&\,  w \sum^n_{\a = 1} c_\a   \sum_i S_{\a - 1;i} (X^{i\bar i})^2 \bpartial_i \p_i \phi - C_1 w \sum^n_{\a = 1} c_\a \sum_{i} S_{\a - 1;i} (X^{i\bar i})^2  - C_3 .
\end{aligned}
\end{equation}
Consequently,
\begin{equation}
\label{J:C2-inequality-6}
\begin{aligned}
	&\, A w \Bigg(\sum^n_{\a = 1} c_\a   \sum_i S_{\a - 1;i} (X^{i\bar i})^2 (v_{i\bar i} - u_{i\bar i}) + \frac{1}{c} - \sum^n_{\a = 1} c_\a S_\a\Bigg) \\
	\leq&\, C_1 w \sum^n_{\a = 1} c_\a \sum_{i} S_{\a - 1;i} (X^{i\bar i})^2  + C_3 .
\end{aligned}
\end{equation}

Similar to Lemma~\ref{lemma-inequality}, we can show that there are constants $N , \theta > 0$ such that when $w \geq N$ ,
\begin{equation}
\begin{aligned}
	&\, \sum^n_{\a = 1} c_\a \sum_i S_{\a - 1;i} (X^{i\bar i})^2 (v_{i\bar i} - u_{i\bar i}) -\frac{1}{c} + \sum^n_{\a = 1} c_\a S_\a  \\
	\geq&\, \theta + \theta \sum^n_{\a = 1} c_\a \sum_i S_{\a - 1;i} (X^{i\bar i})^2 .
\end{aligned}
\end{equation}
If $w \geq N$, from \eqref{J:C2-inequality-6}
\begin{equation}
\label{J:C2-inequality-7}
\begin{aligned}
	&\, A w \Bigg(\theta + \theta \sum^n_{\a = 1} c_\a \sum_i S_{\a - 1;i} (X^{i\bar i})^2\Bigg) \\
	\leq&\, C_1 w \sum^n_{\a = 1} c_\a \sum_{i} S_{\a - 1;i} (X^{i\bar i})^2  + C_3 .
\end{aligned}
\end{equation}
Choosing $A > \frac{C_1}{\theta}$, we derive a bound for $w$.

\end{proof}

\medskip
\subsection{Uniform estimate}

\begin{theorem}
\label{J:theorem-kahler-c0}
Under the assumption of Theorem~\ref{introduction:theorem-kahler-J}, there exists a uniform constant $C$ such that
\begin{equation}
	\sup_M u(x,t) - \inf_M u(x,t) \leq C.
\end{equation}
\end{theorem}

\begin{proof}
Different from Section~\ref{long}, we shall use the well known functional $I$ defined by
\begin{equation}
	I = J_0 = \int^1_0 \int_M \frac{\p v}{\p s} \chi^{n }_v ds ,
\end{equation}
for any path $v(s) \subset \mathcal{H}$ connecting $0$ and $u$. Then
\begin{equation}
\label{J:theorem-kahler-proof-I}
\begin{aligned}
	I (u(T)) &= \int^T_0 \int_M \p_t u \chi^n_u dt = 0.
\end{aligned}
\end{equation}
It is shown by Weinkove~\cite{Weinkove06} that equality~\eqref{J:theorem-kahler-proof-I} implies
\begin{equation}
\label{J:theorem-kahler-proof-I-inequality}
	0 \leq \sup_M u(x,t) \leq C_1 - C_2 \inf_M u(x,t).
\end{equation}
Therefore, we only need to prove the lower bound of $u(x,t)$. If such a lower bound does not exist, then we can choose a sequence $t_i \rightarrow \infty$ such that
\begin{equation}
\label{J:theorem-kahler-proof-1}
	\inf_M  u(x,t_i) = \inf_{t \in [0,t_i]} \inf_M  u(x,t) ,
\end{equation}
and
\begin{equation}
\label{J:theorem-kahler-proof-2}
	\inf_M  u(x,t_i) \rightarrow - \infty .
\end{equation}
By theorem~\ref{J:theorem-C2-estimate},
\begin{equation}
\begin{aligned}
	w(x,t_i) &\leq C e^{A (u(x,t_i) - \inf_{M \times [0,t_i]} u)} \\
	&= C e^{A (u(x,t_i) - \inf_M u(x,t_i))}.
\end{aligned}
\end{equation}
As shown in \cite{TWv10a}, it follows that
\begin{equation}
	\sup_M u(x,t_i) - \inf_M u(x,t_i) \leq C
\end{equation}
for some positive constant $C$, which contradicts our assumption.

\end{proof}

\medskip

\subsection{Lower bound and properness}
\label{lower}

We consider
\begin{equation}
\label{long:definition-J-new-normalization}
	\hat J_\chi(u)  := \sum^n_{\a = 1} b_\a J_{\a}(u) - c I(u), \qquad
	c = \frac{\int_M \chi^n}{\sum^n_{\a = 1} b_\a \int_M \chi^{n - \a} \wedge \omega^\a } .
\end{equation}
The generalized $J$-flow is the gradient flow for $\hat J_\chi$.

Choosing a path $\phi(t)$, we have
\begin{equation}
\label{lower:functional-derivative-1}
\begin{aligned}
	\frac{d \hat J_\chi(\phi)}{d t} 
	&= \int_M \dot{\phi} \Bigg( \sum^n_{\a = 1} b_\a \chi^{n - \a}_\phi \wedge \omega^\a - c \chi^n_\phi \Bigg)
\end{aligned}
\end{equation}
and
\begin{equation}
\label{lower:functional-derivative-2}
\begin{aligned}
	\frac{d^2 \hat J_\chi(\phi)}{d t^2}
	&=  \int_M \ddot{\phi} \Bigg( \sum^n_{\a = 1} b_\a \chi^{n - \a}_\phi \wedge \omega^\a - c \chi^n_\phi \Bigg) \\
	& - \frac{1}{2}\int_M \sqrt{-1} \p\dot{\phi}\wedge\bpartial\dot{\phi} \wedge \Bigg(\sum^n_{\a = 1} b_\a (n - \a) \chi^{n - \a - 1}_\phi \wedge \omega^\a - c n \chi^{n - 1}_\phi \Bigg) .
\end{aligned}
\end{equation}

Similar to the work of Mabuchi~\cite{Mabuchi86}, the class of functional $\hat J_\chi$ has 1-cocyle condition in $\mathcal{H}$ by the independence of path. Together with \eqref{lower:functional-derivative-1} and \eqref{lower:functional-derivative-2}, we have the following theorem.
\begin{theorem}
\label{lower:theorem-uniqueness}
There is at most one K\"ahler metric $\chi' \in [\chi]$ such that
\begin{equation}
\label{lower:equation-1}
\sum^n_{\a = 1} b_\a \chi'^{n - \a} \wedge \omega^\a = c \chi'^n .
\end{equation}
In addition, if \eqref{lower:equation-1} holds true for $\chi' = \chi_u$, then $\hat{J}_\chi$ has a minimum value at $u$.
\end{theorem}
\begin{proof}
The 1-cocycle condition tells us that $\hat J_\chi (v) = \hat J_\chi (u) + \hat J_{\chi_u} (v)$. 
If there are two metrics $\chi' = \chi_u$ and $\chi'' = \chi_{u'}$ satisfying equation~\eqref{lower:equation-1}. Applying \eqref{lower:functional-derivative-1} and \eqref{lower:functional-derivative-1} to $\hat J_{\chi_u}$ with $\phi (t) = t ( u' - u)$, 
\begin{equation}
	\frac{d \hat J_{\chi'}(\phi)}{d t} \Bigg|_{t = 0}=\frac{d \hat J_{\chi'}(\phi)}{d t} \Bigg|_{t = 1} =  0,
\end{equation}
and
\begin{equation}
\begin{aligned}
	\frac{d^2 \hat J_\chi(\phi)}{d t^2}
	&=  - \frac{1}{2}\int_M \sqrt{-1} \p\dot{\phi}\wedge\bpartial\dot{\phi} \wedge \Bigg(\sum^n_{\a = 1} b_\a (n - \a) \chi^{n - \a - 1}_\phi \wedge \omega^\a - c n \chi^{n - 1}_\phi \Bigg) \\
	&\geq \delta \int_M \sqrt{-1} \p (u - u')\wedge\bpartial(u - u') \wedge \chi^{n - 1}_\phi .
\end{aligned}
\end{equation}
So it must be that $u - u'$ is constant.

For any path $\phi$, we have
\begin{equation}
	\frac{d \hat J_{\chi'}(\phi)}{d t} \Bigg|_{t = 0}= 0,
\end{equation}
and 
\begin{equation}
	\frac{d^2 \hat J_{\chi'}(\phi)}{d t^2} \Bigg|_{t = 0} > 0 .
\end{equation}
So $\hat J_\chi (v) = \hat J_\chi (u) + \hat J_{\chi'} (v) \geq  \hat J_\chi (u)$

\end{proof}
An alternative approach is probably the method by Chen~\cite{Chen00a,Chen00b} using the $C^{1,1}$ geodesics in $\mathcal{H}$.

We recall the definition of properness in \cite{CollinsSzekelyhidi2014a}.
\begin{definition}
\label{lower:definition-proper}
We say that $\hat J_\chi$ is proper, if there are constants $C, \delta > 0$ such that for any $\chi_u > 0$, 
\begin{equation}
	\hat J_\chi (u) \geq - C + \delta \int_M u (\chi^n - \chi^n_u).
\end{equation}
\end{definition}

A corollary follows from Theorem~\ref{introduction:theorem-kahler} (or Theorem~\ref{introduction:theorem-kahler-J} ), Theorem~\ref{lower:theorem-uniqueness} and the perturbation method in \cite{CollinsSzekelyhidi2014a}.
\begin{theorem}
Suppose that the cone condition~\eqref{introduction:cone-condition} holds true for $\psi = c$. Then the functional $\hat J_\chi$ is bounded from below and then proper.
\end{theorem}


\bigskip
























\begin{thebibliography}{99}
\bibitem{Aubin78}
T. Aubin,
{\em \'Equations du type Monge-Amp\`ere sur les vari\'et\'es k\"ahl\'eriennes compactes}, 
(French)  Bull. Sci. Math. (2) {\bf 102} (1978), 63--95.

\bibitem{CNS3}
L. A. Caffarelli, L. Nirenberg and J. Spruck,
{\em The Dirichlet problem for nonlinear second-order elliptic equations III: Functions of eigenvalues of the Hessians},
{Acta Math.} {\bf 155} (1985), 261--301.

\bibitem{Calabi56}
E. Calabi,
{\em The space of K\"ahler metrics},
Proc. ICM, Amsterdam, Vol. {\bf 2}  (1954), 206--207.

\bibitem{Calabi57}
E. Calabi,
{\em On K\"ahler manifolds with vanishing canonical class},
in {\em Algebraic geometry and topology: A symposium in honor of S. Lefschetz}, 78--89. Princeton University Press, 1957.

\bibitem{Cao85}
H.-D. Cao,
{\em Deformation of K\"ahler metrics to K\"ahler-Einstein metrics on compact K\"ahler manifolds},
Invent. Math. {\bf 81} (1985), 359--372.

\bibitem{Chen00a}
X.-X. Chen,
{\em The space of K\"ahler metrics},
Journ. Diff. Geom. {\bf 56} (2000). 189--234.

\bibitem{Chen00b}
X.-X. Chen,
{\em On the lower bound of the Mabuchi energy and its application},
Int. Math. Res. Notices \textbf{12}  (2000),  607--623.

\bibitem{Chen04}
X.-X. Chen,
{\em A new parabolic flow in K\"ahler manifolds},
Comm. Anal. Geom. {\bf 12} (2004), 837--852.

\bibitem{Cherrier87}
P. Cherrier,
{\em Equations de Monge-Amp\`ere sur les vari\'et\'es hermitiennes compactes},
Bull. Sci. Math. {\bf 111} (1987), 343--385.

\bibitem{CollinsSzekelyhidi2014a}
T. C. Collins and G. Sz\'ekelyhidi,
{\em Convergence of the $J$-flow on toric manifolds},
preprint, arXiv:1412.4809.

\bibitem{Donaldson99a}
S. K. Donaldson,
{\em Moment maps and diffeomorphisms},
Asian J. Math. {\bf 3} (1999), 1--16.


\bibitem{Evans82}
L. C. Evans,
{\em Classical solutions of fully nonlinear, convex, second order elliptic equations},
Comm. Pure Appl. Math. {\bf 35} (1982), 333--363.

\bibitem{FLM11}
H. Fang, M.-J. Lai and X.-N. Ma,
{\em On a class of fully nonlinear flows in K\"ahler geometry},
J. Reine Angew. Math. {\bf 653} (2011), 189--220.



\bibitem{FL12}
 H. Fang and M.-J. Lai,
{\em On the geometric flows solving K\"ahlerian inverse $\sigma_k$ equations},
Pacific J. Math. {\bf 258}, no. 2  (2012), 291--304.


\bibitem{FL13}
H. Fang and M.-J. Lai,
{\em Convergence of general inverse $\sigma_k$-flow on K\"ahler manifolds with Calabi ansatz},
Trans. Amer. Math. Soc. {\bf 365}, no. 12 (2013), 6543--6567.



\bibitem{Gill11}
M. Gill,
{\em Convergence of the parabolic complex Monge-Amp\`ere equation on compact Hermitian manifolds},
Comm. Anal. Geom. {\bf 19}, no. 2 (2011), 277--304.


\bibitem{Guan2014a}
B. Guan,
{\em Second order estimates and regularity for fully nonlinear elliptic equations on Riemannian manifolds},
Duke Math. J. {\bf 163} (2014), 1491--1524.


\bibitem{GL10}
B. Guan and  Q.  Li,
{\em Complex Monge-Amp\`ere equations and totally real submanifolds},
Adv. Math. {\bf 225} (2010), 1185--1223.


\bibitem{GSun12}
B.~ Guan and W.~Sun,
{\em On a class of fully nonlinear elliptic equations on Hermitian manifolds},
Calc. Var. PDE, DOI 10.1007/s00526-014-0810-1.




\bibitem{GLZ}
P.-F. Guan, Q. Li and X. Zhang, 
{\em A uniqueness theorem in K\"ahler geometry}, 
Math.Ann. {\bf 345} (2009), 377--393.


\bibitem{Krylov82}
N. V. Krylov,\,
{\em Boundedly nonhomogeneous elliptic and parabolic equations},
Izvestiya Ross. Akad. Nauk. SSSR {\bf 46} (1982), 487--523.

\bibitem{LejmiSzekelyhidi13}
M. Lejmi and G. Sz\'ekelyhidi,
{\em The J-flow and stability},
preprint, arXiv:1309.2821.

\bibitem{LiShiYao2013}
H.-Z. Li, Y.-L. Shi and Y. Yao,
{\em A criterion for the properness of the $K$-energy in a general K\"ahler class},
Math. Ann., DOI 10.1007/s00208-014-1073-z. 





\bibitem{Mabuchi86}
T. Mabuchi,
{\em $K$-ernergy maps integrating Futaki invariants},
Tohoku Math. J. {\bf 38} (1986), no. 4, 575--593.

\bibitem{PhongSturm10}
D. H. Phong and J. Sturm,
{\em The Dirichlet problem for degenerate complex Monge-Amp\`ere equations},
Comm. Anal. Geom. {\bf 18} (2010), no. 1, 145--170.

\bibitem{SW08}
J. Song and B. Weinkove,
{\em On the convergence and singularities of the J-flow with applications
to the Mabuchi energy},
Comm. Pure Appl. Math. {\bf 61} (2008), 210--229.


\bibitem{SunDissertation}
W. Sun,
{\em On a class of complex Monge-Amp\`ere type equations on Hermitian manifolds},
Doctoral dissertation, The Ohio State University (2013).

\bibitem{Sun2013e}
W. Sun,
{\em On a class of fully nonlinear elliptic equations on closed Hermitian manifolds},
preprint,  arXiv:1310.0362.

\bibitem{Sun2013p}
W. Sun,
{\em Parabolic complex Monge-Amp\`ere type equations on closed Hermtian manifolds},
preprint, arXiv:1311.3002.

\bibitem{Sun2014e}
W. Sun,
{\em On a class of fully nonlinear elliptic equations on closed Hermitian manifolds II: $L^\infty$ estimate},
preprint, arXiv:1407.7630.

\bibitem{Sun2014u}
W. Sun,
{\em On uniform estimate of complex elliptic equations on closed Hermitian manifolds},
preprint, arXiv:1412.5001.


\bibitem{Sun2014g}
W. Sun,
{\em Generalized complex Monge-Amp\`ere type equations on closed Hermitian manifolds},
preprint, arXiv:1412.8192.


\bibitem{Szekelyhidi2014b}
G. Sz\'ekelyhidi,
{\em Fully non-linear elliptic equations on compact Hermitian manifolds},
preprint, arXiv:1501.02762.

\bibitem{TWv10a}
V. Tosatti and B. Weinkove,
{\em Estimates for the complex Monge-Amp\`ere equation on Hermitian and
balanced manifolds},
Asian J. Math. {\bf 14} (2010), 19--40.

\bibitem{TWv10b}
V. Tosatti and B. Weinkove,
{\em The complex Monge-Amp\`ere equation on compact Hermitian manifolds},
J. Amer. Math. Soc. {\bf 23} (2010), 1187--1195.

\bibitem{TWv11}
V. Tosatti and B. Weinkove,
{\em On the evolution of a Hermitian metric by its Chern-Ricci form},
J. Differential Geom. {\bf 99} (2015), no.1, 125--163.






\bibitem{Wang1992a}
L.-H. Wang,
{\em On the regularity theory of fully nonlinear parabolic equations. I.},
Comm. Pure Appl. Math. {\bf 45} (1992), no. 1, 27--76.

\bibitem{Wang1992b}
L.-H. Wang,
{\em}
Comm. Pure Appl. Math. {\bf 45} (1992), no. 2, 141--178.

\bibitem{Weinkove04}
B. Weinkove,
{\em Convergence of the J-flow on K\"ahler surfaces},
Comm. Anal. Geom. {\bf 12} (2004), 949--965.

\bibitem{Weinkove06}
B. Weinkove,
{\em On the J-flow in higher dimensions and the lower boundedness
of the Mabuchi energy},
J. Differential Geom. {\bf 73} (2006), 351--358.

\bibitem{Yau78}
S.-T. Yau,
{\em On the Ricci curvature of a compact K\"ahler manifold and the complex
Monge-Amp\`ere equation. I.}
Comm. Pure Appl. Math. {\bf 31} (1978), 339--411.


\end{thebibliography}
\end{document}